\documentclass[10pt]{article}
\usepackage[left=3cm,top=3cm,right=3cm,bottom=3cm]{geometry}
\usepackage{amsfonts,amssymb,amsmath,latexsym,indentfirst}
\usepackage{fontenc,enumerate}
\usepackage{graphics}
\usepackage{color,graphicx}
\usepackage[pdftex]{hyperref}
\usepackage{latexsym,wasysym,mathrsfs,enumerate}
\DeclareMathAlphabet{\mathpzc}{OT1}{pzc}{m}{it}

\definecolor{darkred}{rgb}{.0,.0,.8}
\definecolor{darkgreen}{rgb}{.0,0,0.8}
\definecolor{darkredd}{rgb}{0,0,.8}
\hypersetup{
 colorlinks,%
       citecolor=darkgreen,%
    filecolor=magenta,%
    linkcolor=darkredd,%
    urlcolor=darkred,%
}
\numberwithin{equation}{section}
\newtheorem{proposition}{Proposition}[section]
\newtheorem{remark}{Remark}[section]
\newtheorem{definition}{Definition}[section]
\newtheorem{lemma}{Lemma}[section]
\newtheorem{theorem}{Theorem}[section]
\newtheorem{corollary}{Corollary}[section]

\newtheorem{obs}{Observation}[section]

\newcommand{\N}{\ensuremath{{\mathbb{N}} }}

\newcommand{\<}{\langle}
\renewcommand{\>}{\rangle}
\newcommand{\ga}{\gamma}
\newcommand{\eps}{\epsilon}

\newcommand{\si}{\sigma}

\newcommand{\al}{\alpha}

\newcommand{\rr}{{\mathbb R}}

\newcommand{\LL}{\Lambda}
\newcommand{\ff}{\varphi}
\newcommand{\vff}{\vec{\varphi}}
\newcommand{\vp}{\vec{\psi}}
\newcommand{\bvp}{B_{\vec{\psi}}}
\newcommand{\vdu}{\vec{u}_0}
\newcommand{\vu}{\vec{u}}

\newcommand{\what}{\widehat}
\newcommand{\dd}{\;{\rm d}}
\newcommand{\ee}{{\rm e}}
\newcommand{\ii}{{\rm i}}

\newcommand{\hb}{ \mathbb{H}}
\newcommand{\kk}{\mathscr{K}}
\newcommand{\kb}{\mathbb{K}}
\newcommand{\n}{\mathscr{N}}
\newcommand{\ddd}{\mathscr{D}}
\newcommand{\p}{\mathscr{P}}

\newcommand{\uu}{\mathscr{U}}
\newcommand{\vtt}{\vartheta}
\newcommand{\g}{\mathscr{G}(\beta,c)}
\newcommand{\vg}{\vec{\mathscr{G}}(\beta,c)}
\newcommand{\x}{\mathscr{X}}

\newcommand{\lt}{{L^2(\mathbb{R})}}

\newcommand{\para}{\hspace*{0.25in}}

\newcommand{\fim}{\hfill$\square$\vskip 20pt}

\newcommand{\q}{\quad}

\newcommand{\proof}{\noindent\textbf{Proof.}\quad}

\author{{\bf Amin Esfahani}\\  {\small School of Mathematics and Computer Science} \\
{\small Damghan University}\\ {\small Damghan, Postal Code 36716-41167, Iran}\\
{\small  E-mail: amin@impa.br, saesfahani@du.ac.ir}\vspace{2mm}\\
{\bf Steven Levandosky}\\
{\small Mathematics and Computer Science Department}\\
{\small College of the Holy Cross, Worcester, MA 01610} \\
{\small  E-mail: spl@mathcs.holycross.edu}}

\title{Stability of solitary waves for the generalized higher-order Boussinesq equation
\footnotetext{Mathematical subject classification: 35Q35, 76B55, 76U05, 76B25, 35B35}
\footnotetext{Keywords: Boussinesq equation, solitary waves, stability}}

\date{}

\begin{document}
\maketitle

\begin{abstract}
This work studies the stability of solitary waves of a class of sixth-order Boussinesq equations.
\end{abstract}

\section{Introduction}\label{S:intro}

In this work we study the generalized sixth-order Boussinesq (GSBQ) equation \cite{cmv,esfahani_farah,esfahani_farah_wang}
\begin{equation}\label{E:B6}
u_{tt} = u_{xx} + \beta u_{xxxx} + u_{xxxxxx} - (f(u))_{xx}
\end{equation}
where $f\in C^2$ is homogeneous of degree $p\geq2$. Neglecting the sixth-order term, equation \eqref{E:B6} becomes a generalization of the classical
Boussinesq equations
\begin{equation}\label{E:B4}
u_{tt} = u_{xx} + \beta u_{xxxx} - (f(u))_{xx}, \qquad\beta=\pm1,
\end{equation}
Equation \eqref{E:B4} was originally derived by Boussinesq \cite{boussinesq} in his study of nonlinear, dispersive wave propagation. We should remark that it was the first equation proposed in the literature to describe this kind of physical phenomena. Equation \eqref{E:B4} was also used by Zakharov  \cite{zakharov} as a model of nonlinear string and by Falk \emph{et al} \cite{fls} in their study of shape-memory alloys.

When $\beta = 1$, equation\eqref{E:B4} is called ``bad" Boussinesq equation, while \eqref{E:B4} with $\beta=-1$,
\begin{equation}\label{E:B4_good}
u_{tt} = u_{xx} - u_{xxxx} - (f(u))_{xx},
\end{equation}
is called ``good" Boussinesq equation. Given certain conditions on $f$, \eqref{E:B4_good} possesses special
traveling-wave solutions with finite energy. Indeed, \eqref{E:B4_good} can be written as the system of equations
\begin{equation}\label{E:B4_system}
\begin{split}
  u_t &= v_x\\
  v_t &= (u - u_{xx} - f(u))_x
\end{split}
\end{equation}
By a solitary wave solution of \eqref{E:B4_system}, we mean a traveling-wave solution of the form
$\vec\varphi(x-ct)$, vanishing at infinity, where $c $ is the speed of wave propagation. It was shown in \cite{bona-sachs,liu1} that these solutions are of the form $\vec\varphi= (\varphi,-c\varphi)$ so that they must satisfy
\begin{equation}
(1- c^2)\ff -\ff'' - f(\ff) = 0.
\end{equation}

Bona and Sachs in \cite{bona-sachs} proved that the solitary waves of \eqref{E:B4_system} are stable under an appropriate convexity condition.  Liu \cite{liu1,liu2} showed the nonlinear instability of solitary waves of \eqref{E:B4_system}. His proof was based on a modification of the general argument of \cite{gss}.

Equation \eqref{E:B6} can be also written as the following system of equations
\begin{equation}\label{E:B6_system}
\begin{split}
  u_t &= v_x\\
  v_t &= (u +\beta u_{xx}+u_{xxxx} - f(u))_x
\end{split}
\end{equation}
If we put the solitary wave form $\ff(x - ct)$ into \eqref{E:B6}, we obtain
\begin{equation}\label{E:B6_solitary}
    (1-c^2)\ff+\beta\ff''+\ff''''-f(\ff)=0.
\end{equation}
It is worth noting that the solitary wave solutions of equation \eqref{E:B6_solitary} have been investigated numerically and the two classes of subsonic solutions corresponding to the sign of $\beta$ have been obtained, more precisely, the monotone shapes and the shapes with oscillatory tails \cite{cmv}.

The system \eqref{E:B6_system} has the conserved quantities
\begin{align}
  E(u,v)&=\int_\rr \frac12(u_{xx}^2-\beta u_x^2+u^2+v^2)-F(u)\dd x\\
  Q(u,v)&=\int_\rr uv\dd x
\end{align}
We also note that, at least formally, the quantity
\[
\int_\rr u\partial_x^{2k}v\dd x
\]
is conserved for any positive integer $k$. If $\vff$ is a solution of the solitary wave equation \eqref{E:B6_solitary}, then $\vec\ff=(\ff,-c\ff)$ satisfies
    \[
    E'(\vff)+cQ'(\vff)=\vec 0,
    \]
so solitary waves are critical points of the action
\begin{equation}\label{E:action}
  L(u,v)=E(u,v)+cQ(u,v).
\end{equation}
Our aim here is to study the stability of solitary waves of \eqref{E:B6}.

This paper is organized as follows. In Section \ref{S:existence}, we consider the properties of ground state solitary wave solutions. The solitary wave equation \eqref{E:B6_solitary} is a fourth-order elliptic equation, and is identical, after a rearrangement of parameters, to the solitary wave equation that arises in the study of the fifth-order KdV equation. The variational, regularity, and decay properties of this equation were considered in \cite{levandosky2}, so we refer to this work for several results. In Section \ref{S:stability} we prove the main stability result, Theorem \ref{T:stability}, which states that the set of ground state solitary waves is stable if $d''(c)>0$, where $d$ is defined by equation \eqref{E:d_definition}. In Section \ref{S:instability} we prove the main instability result, Theorem \ref{T:instability}, which states that a given ground state is orbitally unstable if there exists an ``unstable direction''. In Theorem \ref{T:instability_d} we show that such an unstable direction exists provided $d''(c)<0$. Using a different choice of unstable direction, we also derive in Theorem \ref{T:instability_criteria} explicit conditions on $p$, $\beta$ and $c$ that imply orbital instability. Section \ref{S:d_properties} is devoted to establishing further properties of the function $d$. We first show that when $f(u)=|u|^{p-1}u$ for $p<5$, there exist $c$ near $c_*$ such that $d''(c)>0$. See Theorem \ref{T:d_bound} and Corollary \ref{C:stability}. We then derive in Theorem \ref{T:d_scaling} the main scaling identity satisfied by $d$, and use it to prove that $d''(c)$ may change sign at most once along each semi-ellipse in the $(\beta,c)$-plane. Finally, in Section \ref{S:numerical}, we outline the numerical method used to compute the function $d$, and present the results of these numerical calculations. The main conclusions that can be drawn from these results are found in Observation \ref{O:numerical_results}.

\section*{Notations}
For each $r\in\rr$, we define the translation operator by $\tau_ru= u(\cdot+r)$.

Given a solitary wave $\vff$ of \eqref{E:B6_system},
the orbit of $\vff$ is defined by the set $\mathcal{O}_{\vff} = \{\tau_r\vff;\;r\in\rr\}$.

We shall denote by $\widehat{g}$ the Fourier transform of $g$, defined as
\[
\widehat{g}(\zeta)=\int_{\rr}\;g(\omega)\ee^{-\ii \omega\cdot\zeta}\;\dd\omega.
\]
For $s\in\rr$ and $1\leq p\leq\infty$, we denote by $H^{s,p}(\rr)$, the Bessel potential space defined by $H^{s,p}(\rr)=\Lambda^{-s}L^p(\rr)$,
with respect to the norm
\[
\|g\|_{H^{s,p}(\rr)}
=\|\Lambda^sg\|_{L^p(\rr)},
\]
where $\LL^s=(I-\partial_x^2)^{s/2}$.
In particular, we define  the nonhomogeneous Sobolev space $H^s(\rr)=H^{s,2}(\rr)$.
Let $\x$ be the space defined by
\[
\x=H^2(\rr)\times L^2(\rr),
\]
 with the norm
 \[
{ \|\vu\|}_\x={ \|(u,v)\|}_\x={\|u\|}_{H^2(\rr)}+{\|v\|}_{L^2(\rr)}.
 \]

For any positive numbers $a$ and $b$, the notation $a \lesssim b$ means that there exists a
positive (harmless) constant $\mathpzc{k}$ such that $a\leq\mathpzc{k}  b$. We also use $a\sim b$ when $a\lesssim b$ and $b\lesssim a$.

\section{Existence of Solitary Waves}\label{S:existence}

Solutions of the solitary wave equation \eqref{E:B6_solitary} may be shown to exist via the following variational problem. Define
\begin{align}
  I(u)&=\int_\rr u_{xx}^2-\beta u_x^2+(1-c^2)u^2\dd x\\
  K(u)&=(p+1)\int_\rr F(u)\dd x
\end{align}
where $F'=f$ and $F(0)=0$. When $c^2<1$ and $\beta<\beta_*=2\sqrt{1-c^2}$ (equivalently when $\beta<2$ and $c^2<c_\ast^2$, where $c_\ast=\sqrt{1-\beta_+^2/4}$ and $\beta_+=\max\{\beta,0\}$), the functional $I$ is coercive in the sense that
    \begin{equation}\label{E:I_coercive}
    I(u)\geq C(\beta,c)\|u\|_{H^2(\rr)}^2
    \end{equation}
where
    \[
    C(\beta,c)>\left\{
    \begin{array}{cl}
      1-c^2&\beta\leq 0\\
      1-c^2-\frac12\beta\sqrt{1-c^2}&\beta>0
    \end{array}
    \right\}>0.
    \]
Since $K(u)\leq C\|u\|_{H^2(\rr)}^{p+1}$, it follows that for $\lambda>0$ we have
    \[
    M_\lambda=\inf\{I(u)\mid u\in H^2(\rr), K(u)=\lambda\}>0.
    \]
We say that a sequence $u_k$ is a {\em minimizing sequence} if $K(u_k)\to\lambda>0$ and $I(u_k)\to M_\lambda$. The following result is a consequence of the concentration-compactness theorem, and was shown in \cite{levandosky2} for a more general class of homogeneous nonlinearities (see also \cite{esfahani_levandosky,kara-mc}).
\begin{theorem}\label{T:existence}
Fix $p>1$. Suppose $c^2<1$ and $\beta<\beta_*$. If $u_k$ is a minimizing sequence for some $\lambda>0$, then there exists a subsequence $u_{k_j}$, scalars $y_j$ and $\psi\in H^2(\rr)$ such that $u_{k_j}(\cdot-y_j)\to\psi$ in $H^2(\rr)$.
\end{theorem}
Since the function $\psi$ achieves the minimum $M_\lambda$ it satisfies the Euler-Lagrange equation
\[
(1-c^2)\psi+\beta\psi''+\psi''''=\mu f(\psi),
\]
for some multiplier $\mu$. Multiplying this equation by $\psi$ and integrating over $\rr$, it follows that
$M_\lambda=I(\psi)=\mu(p+1)\lambda$, so $\mu>0$. Thus $\varphi=\mu^{1/(p-1)}\psi$ is a solution of the solitary wave equation \eqref{E:B6_solitary}. Such solutions are referred to as {\em ground states} and, by the homogeneity of $F$, achieve the minimum
\[
m(\beta,c)=\inf\left\{\frac{I(u)}{K(u)^{2/(p+1)}}: u\in H^2(\rr),u\neq0\right\}.
\]
The set of all ground states will be denoted by $\g$. Multiplying the solitary wave equation \eqref{E:B6_solitary} by $\ff$ and integrating gives $I(\ff)=K(\ff)$. Thus the set of ground states is given by
    \begin{equation}\label{E:ground_state_definition}
    \g=\{\ff\in H^2(\rr): I(\ff)=K(\ff)=m(\beta,c)^{\frac{p+1}{p-1}}\}.
    \end{equation}
We shall denote
    \[
    \vg=\{\vff=(\ff,-c\ff)\in\x:\ff\in\g\}.
    \]
As mentioned in the introduction, elements of $\vg$ are critical points of the action $L$ defined by \eqref{E:action}. In fact, elements of $\vg$ are minimizers of $L$ subject to the constraint $P=0$, where
\begin{equation}\label{E:P_definition}
  P(\vec w)=\<L'(\vec w),\vec w\>.
\end{equation}
\begin{theorem}\label{ground}
 Suppose $\beta<\beta_\ast$ and $c^2<1$.
 Let
\begin{equation}\label{E:N_definition}
  \mathscr{N}=\{\vec w\in\x;\;\vec w\neq\vec0,\;P(\vec w)=0\}.
\end{equation}
 The following are equivalent.
 \begin{enumerate}[(i)]
   \item $\vff\in\vg$.
   \item $\vff\in\mathscr{N}$ and $L(\vff)=\inf\{L(\vec w):\vec w\in\mathscr{N}\}$.
 \end{enumerate}
\end{theorem}

\proof The identities that we shall need relating the two variational problems are
\begin{equation}\label{E:LIK}
  L(u,v)=\frac12I(u)-\frac1{p+1}K(u)+\frac12\int_\rr(cu+v)^2\dd x
\end{equation}
and
\begin{equation}\label{E:PIK}
  P(u,v)=I(u)-K(u)+\int_\rr(cu+v)^2\dd x.
\end{equation}
From this it follows that, for any $(u,v)\in\mathscr{N}$, we have $L(u,v)=\frac{p-1}{2(p+1)}K(u)$.

First suppose $\vff\in\g$. Then by definition $I(\ff)=K(\ff)$, so $P(\vff)=0$ and thus
$\vff\in\mathscr{N}$. Denote $\lambda=K(\ff)$.
Then $I(\ff)$ minimizes $I(u)$ over all $u\in H^2(\rr)$ such that $K(u)=\lambda$. Now let $\vec w=(u,v)\in\mathscr{N}$. Then $K(u)>0$, so if we set $\tilde{u}=\alpha u$ where $\alpha=(K(\ff)/K(u))^{\frac1{p+1}}$, then $K(\tilde{u})=K(\ff)$ and consequently $I(\ff)\leq I(\tilde{u})$. Therefore
    \[
    0=P(\ff)=I(\ff)-K(\ff)\leq I(\tilde{u})-K(\tilde{u})=\alpha^2I(u)-\alpha^{p+1}K(u)=\alpha^2(1-\alpha^{p-1})I(u),
    \]
which implies $\alpha\leq 1$. Thus $K(\ff)\leq K(u)$, and it follows that
    \[
    L(\vff)=\frac{p-1}{2(p+1)}K(\ff)\leq \frac{p-1}{2(p+1)}K(u)=L(\vec w).
    \]
Hence (i) implies (ii).

Next suppose $\vff=(\ff,\psi)\in\mathcal{N}$ solves the minimization problem. We need to show that $\ff\in\g$ and $\psi=-c\ff$. Denote $\lambda=K(\ff)>0$ and suppose $u\in H^2(\rr)$ minimizes $I$ subject to the constraint $K(\cdot)=\lambda$. Then
    \[
    u_{xxxx}+\beta u_{xx}+(1-c^2)u=\mu f(u)
    \]
for some $\mu$. Multiplying by $u$ and integrating gives $I(u)=\mu K(u)=\mu\lambda$. Since
    \begin{equation}\label{E:IK_inequality}
      I(u)\leq I(\ff)
      =K(\ff)-\int_\rr(c\ff-\psi)^2\dd x
        \leq K(\ff)\\
      =\lambda,
   \end{equation}
we have $\mu\leq 1$. On the other hand, if we set $\tilde{u}=\mu^{\frac1{p-1}}u$, then $I(\tilde u)=K(\tilde u)$ so if we define $\vec w=(\tilde u,-c\tilde u)$ then we have $\vec w\in\mathscr{N}$. Therefore $L(\vff)\leq L(\vec w)$. Since $\vff\in\mathscr{N}$ we have $L(\vff)=\frac{p-1}{2(p+1)}K(\ff)$ and thus
\begin{align*}
  \frac{p-1}{2(p+1)}K(\ff)&=L(\vff)\\
  &\leq K(\vec w)\\
  &=\frac12I(\tilde u)-\frac1{p+1}K(\tilde u)\\
  &=\frac{p-1}{2(p+1)}I(\tilde u)\\
  &=\frac{p-1}{2(p+1)}\mu^{\frac2{p-1}}I(u)\\
  &\leq \frac{p-1}{2(p+1)}\mu^{\frac2{p-1}}I(\ff)\\
  &\leq \frac{p-1}{2(p+1)}\mu^{\frac2{p-1}}K(\ff).\\
\end{align*}
It then follows that $\mu\geq1$ and thus $\mu=1$. This implies $I(u)=K(u)=\lambda$. But \eqref{E:IK_inequality} then implies that $I(\ff)=K(\ff)=\lambda$ and $\psi=-c\ff$, so we have $\ff\in\g$ and therefore $\vff\in\vg$. This completes the proof.
\fim

As shown in \cite{levandosky2}, solitary waves have the following regularity and decay properties.
\begin{theorem}\label{T:regularity_decay} Suppose $\ff\in H^2(\rr)$ is a weak solution of
\eqref{E:B6_solitary} and that $f\in C^k(\rr)$. Then $\ff$ is a classical solution and
$\ff\in C^{k+4}(\rr)$. Furthermore, $\ff$ decays exponentially as
$|x|\to\infty$.
\end{theorem}

It is noteworthy that regularity and decay properties of the solutions of \eqref{E:B6_solitary} can be obtained by using an argument similar to \cite{esfahani_levandosky} via the following equivalent form of \eqref{E:B6_solitary}
\[
\ff=k\ast f(\ff),
\]
where
\begin{equation}\label{E:kernel_definition}
\what{k}(\xi)=\frac{1}{\xi^4-\beta\xi^2+1-c^2},
\end{equation}
$c^2<1$ and $\beta<\beta_\ast=2\sqrt{1-c^2}$. Using the residue theorem, one obtains the following explicit expressions for $k$.
\begin{equation}\label{E:kernel_formulas}
\kb(x)=\left\{\begin{array}{lll}
\frac{\pi}{\lambda_2^2-\lambda_1^2}\left(\frac{1}{\lambda_1}\ee^{-\lambda_1|x|}-\frac{1}{\lambda_2}\ee^{-\lambda_2|x|}\right),&&\beta<-\beta_\ast,\\ \\
\frac{\pi\sqrt{2}}{\beta_\ast^{3/2}}\left(1+\sqrt{\frac{\beta_\ast}{2}}|x|\right)\;
\ee^{-\sqrt{\frac{\beta_\ast}{2}}|x|},&&\beta=-\beta_\ast,\\ \\
\frac{\pi e^{-\sigma|x|}}{2\sigma\omega(\sigma^2+\omega^2)}\left(\omega\cos(\omega x)+\sigma\sin(\omega|x|)\right),
&&\beta\in(-\beta_\ast,\beta_\ast),
\end{array}
\right.
\end{equation}
where
 \begin{equation}\label{E:lambda_sigma_omega}
 \begin{split}
 \lambda_1&=\sqrt{\frac12\left(-\beta-\sqrt{\beta^2-\beta_\ast^2}\right)}\\
 \lambda_2&=\sqrt{\frac12\left(-\beta+\sqrt{\beta^2-\beta_\ast^2}\right)}\\
 \sigma&=\frac12\sqrt{\beta_\ast-\beta}\\
\omega&=\frac12\sqrt{\beta_\ast+\beta}
\end{split}
\end{equation}

\begin{figure}
\begin{center}
  \scalebox{0.25}{\includegraphics{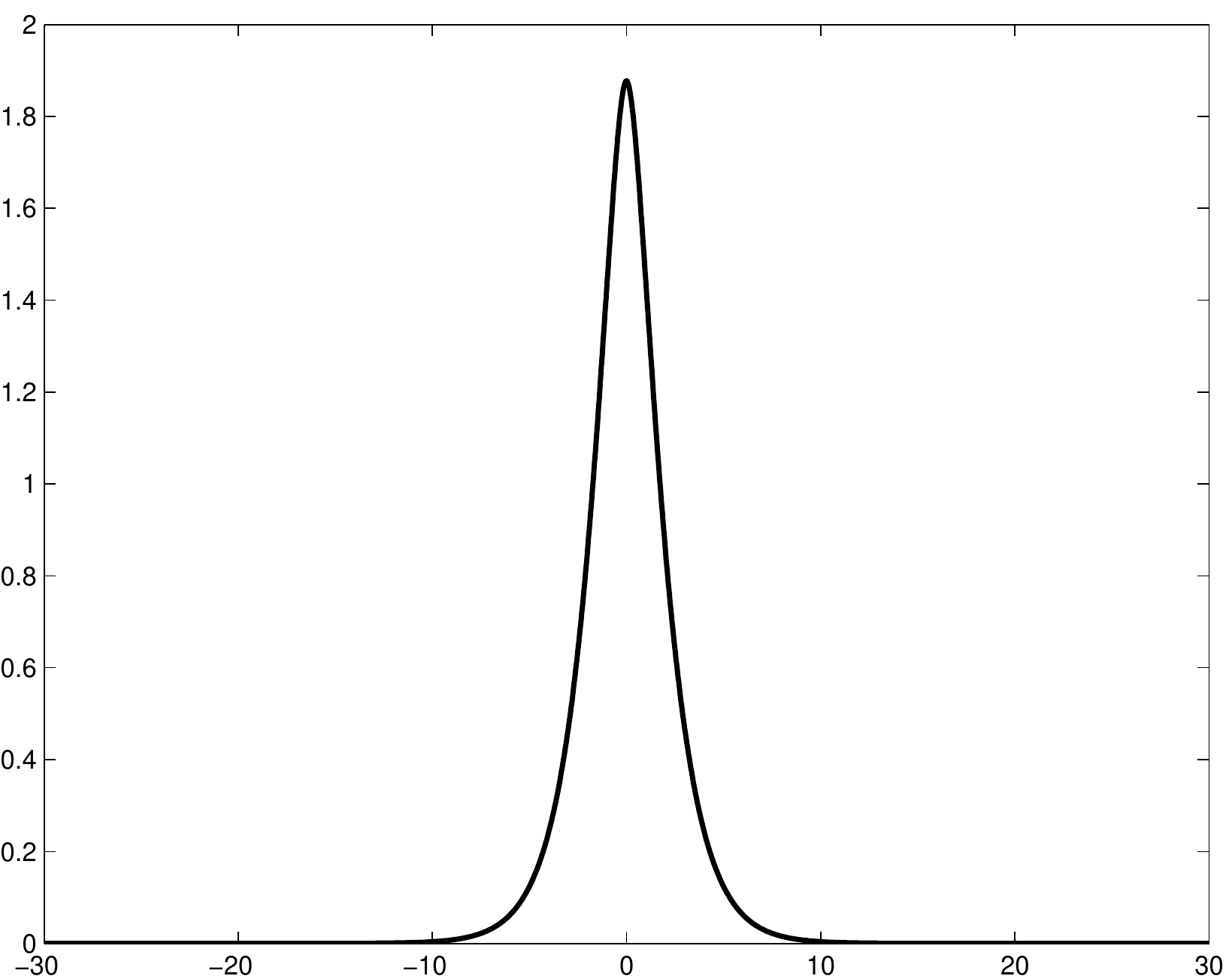}
  \includegraphics{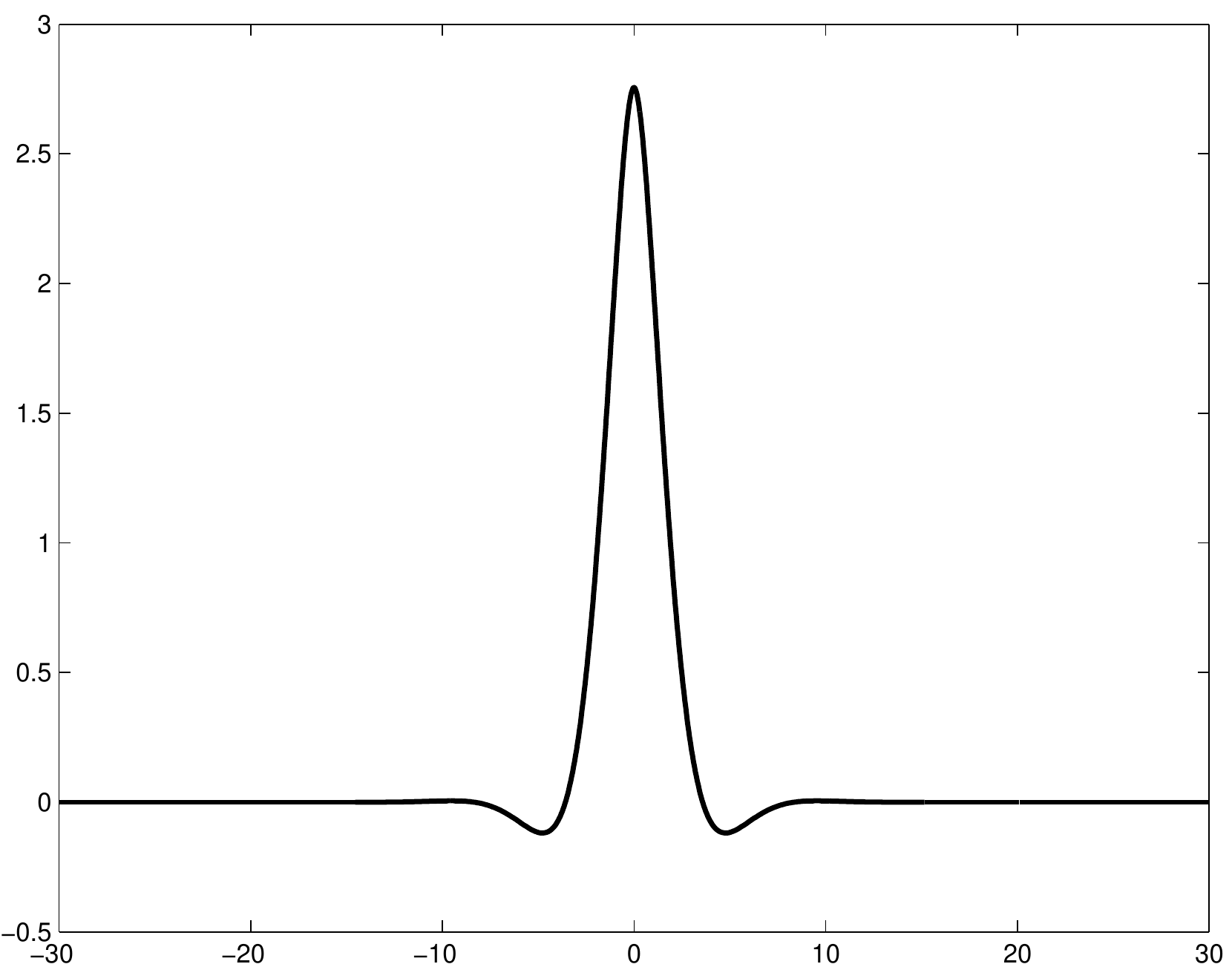}
  \includegraphics{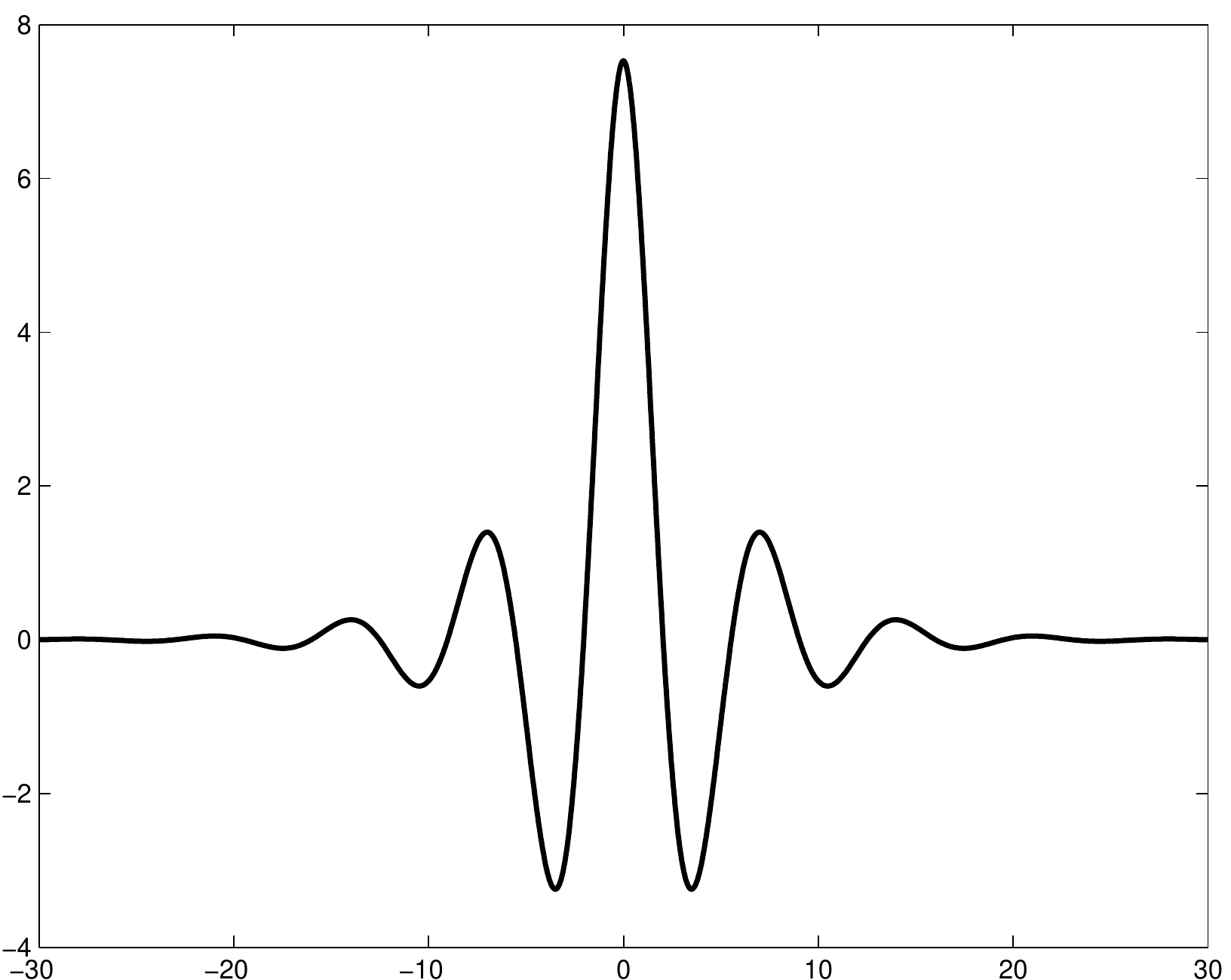}
}
\end{center}
\caption{The kernel $k$, shown here for $c=1/2$, and $\beta=-2$, $\beta=0$ and $\beta=1.5$.}
\end{figure}

One can observe that $k$ oscillates when $\beta\in(-\beta_\ast,\beta_\ast)$; contrary to the case $\beta\leq-\beta_\ast$.
The function $\kb$ may give us an intuition of the properties of the solutions of \eqref{E:B6_solitary}, and is useful in determining the behavior of the function $d$ (see \eqref{E:d_definition}) near the boundary of its domain.

\begin{theorem} There exist no solutions in $H^2(\rr)$ of equation \eqref{E:B6_solitary} if any of the following conditions hold.
\begin{enumerate}[(i)]
  \item $c^2\geq1$ and $\beta<\frac{2\sqrt{(3p+5)(p-1)(c^2-1)}}{p+3}$.
  \item $F(u)\geq0$ for all $u$, $c^2\geq1$ and $\beta\geq0$.
\end{enumerate}
\end{theorem}

\proof Suppose $\ff\in H^2(\rr)$ is a solution of \eqref{E:B6_solitary}.
Multiplying the equation by $x\ff'$ and integrating yields the Pohozaev identity
    \begin{equation}\label{E:pohozaev}
      \int_\rr 3(\ff'')^2-\beta(\ff')^2-(1-c^2)\ff^2+2F(\ff)\dd x=0.
    \end{equation}
The identity $I(\ff)=K(\ff)$ may be written
    \begin{equation}\label{E:I=K}
      \int_\rr (\ff'')^2-\beta(\ff')^2+(1-c^2)\ff^2-(p+1)F(\ff)\dd x=0.
    \end{equation}
Together these give
    \[
    (3p+5)\int_\rr(\ff'')^2\dd x-(p+3)\beta\int_\rr(\ff')^2\dd x-(p-1)(1-c^2)\int\ff^2\dd x=0.
    \]
The term on the left side of this equation will be positive, a contradiction, when condition (i) is satisfied. Next, eliminating the $\ff''$ terms in the equations above gives
    \[
    2\beta\int_\rr(\ff')^2\dd x-4(1-c^2)\int_\rr\ff^2\dd x=-(3p+5)\int_\rr F(\ff)\dd x.
    \]
The conditions in (ii) imply that the left hand side is non-negative and the right hand side is negative.
\fim

\section{Stability}\label{S:stability}
In this section we establish that the set of ground state solitary waves is stable under a suitable convexity condition.

\begin{theorem}[Local Existence]\label{local-h2}
Suppose $p\geq2$. Let $\vdu=(u_0,v_0)\in\x$, then there exists $T>0$ and the unique solution
$\vu=(u,v)\in C([0,T);\x)$ of \eqref{E:B6_system} such that $\vu(0)=\vdu$. Moreover $\vu$ satisfies $E(\vu)=E(\vdu)$, $Q(\vu)=Q(\vdu)$,  $Q_1(\vu)=Q_1(\vdu)$, $Q_2(\vu)=Q_2(\vdu)$ and $Q_3(\vu)=Q_3(\vdu)$ where
\begin{eqnarray}
E(\vu)=E(u,v)&=&\int_\rr \frac{1}{2}(u^2-\beta u_x^2+u_{xx}^2+v^2)-F(u)\dd x,\\
Q(\vu)=Q(u,v)&=&\int_\rr uv\dd x,\\
Q_1(\vu)=Q_1(u,v)&=&\int_\rr u\dd x,\\
Q_2(\vu)=Q_2(u,v)&=&\int_\rr v\dd x,\\
Q_3(\vu)=Q_3(u,v)&=&\int_\rr u\partial_x^{2k}v\dd x,\qquad k\in\N.
\end{eqnarray}
and $F'=f$ and $F(0)=0$.
Furthermore $T=+\infty$, or $T<+\infty$ and
\[
\lim_{t\to T^-}\|\vu\|_{\x}=+\infty.
\]
\end{theorem}
\proof First write the system \eqref{E:B6_system} as
    \[
    \vec w_t=B\vec w+\vec g(\vec w),
    \]
where
    \[
    B=\begin{pmatrix}
      0&\partial_x\\
      \partial_x+\beta\partial_x^3+\partial_x^5&0
    \end{pmatrix}
    \qquad\qquad
    \vec g(\vec w)=(0,-f(u)_x).
    \]
The result then follows by classical semi-group theory \cite{pazy,segal}, once we show that $B$ is the infinitesimal generator of a $C_0$-semigroup of unitary operators on $\x$, and that $\vec g$ is locally Lipschitz on $\x$. Define an inner product $\left\langle\cdot,\cdot\right\rangle_\beta$ on $\x$ by
    \[
    \left\langle (u_1,v_1),(u_2,v_2)\right\rangle_\beta
    =\int_\rr (u_1)_{xx}(u_2)_{xx}-\beta(u_1)_x(u_2)_x+u_1u_2+v_1v_2\dd x.
    \]
Then for and $\vec w=(u,v)\in\x$, we have
\begin{align*}
  \left\langle B\vec w,\vec w\right\rangle_\beta&=
  \int_\rr v_{xxx}u_{xx}-\beta v_{xx}u_x+v_xu+(u_x+\beta u_{xxx}+u_{xxxxx})v\dd x\\
  &=0
\end{align*}
and therefore $B$ is skew adjoint with respect to this inner product. It then follows from Stone's Theorem that $B$ is the infinitesimal generator of a $C_0$-semigroup of unitary operators on $\x$. Now let $\vec w_1, \vec w_2\in\x$. Then
    \begin{align*}
    \|\vec g(\vec w_2)-\vec g(\vec w_1)\|_\x
    &=\|[f(u_1)-f(u_2)]_x\|_{L^2(\rr)}\\
    &=\|f'(u_1)(u_1)_x-f'(u_2)(u_2)_x\|_{L^2(\rr)}\\
    &\leq \|f'(u_1)(u_1-u_2)_x\|_{L^2(\rr)}+\|(u_2)_x[f'(u_1)-f'(u_2)]\|_{L^2(\rr)}
    \end{align*}
To bound the first term, we use the homogeneity of $f$ and the imbedding of $H^2(\rr)$ into $L^\infty(\rr)$ to obtain
    \[
    \|f'(u_1)\|_{L^\infty(\rr)}\leq C\|u_1\|_{L^\infty(\rr)}^{p-1}\leq C\|u_1\|_{H^2(\rr)}^{p-1}
    \leq C\|\vec w_1\|_{\x}^{p-1},
    \]
and thus
    \[
    \|f'(u_1)(u_1-u_2)_x\|_{L^2}\leq C \|\vec w_1\|_{\x}^{p-1}\|(u_1-u_2)_x\|_{L^2(\rr)}
    \leq C \|\vec w_1\|_{\x}^{p-1}\|\vec w_1-\vec w_2\|_\x.
    \]
For the second term, we again use the homogeneity of $f$ and the imbedding $H^1(\rr)$ into $L^\infty(\rr)$ to find
    \begin{align*}
    \|(u_2)_x[f'(u_1)-f'(u_2)]\|^2_{L^2(\rr)}&\leq C\|u_2\|_{H^2(\rr)}(\|u_1\|_{H^2(\rr)}+\|u_2\|_{H^2(\rr)})^{p-2}\|u_1-u_2\|_{L^2(\rr)}\\
    &\leq C\|\vec w_2\|_\x(\|\vec w_1\|_\x+\|\vec w_2\|_\x)^{p-2}\|\vec w_1-\vec w_2\|_\x.
    \end{align*}
Hence $\vec g$ is locally Lipschitz on $\x$, and the proof of local existence is complete. The conservation laws then follow by differentiating each quantity with respect to $t$ and using the system \eqref{E:B6_system}.
\fim

\begin{definition} We say that a subset $S\subseteq\x$ is $\x$-stable if for every $\epsilon>0$ there exists some $\delta>0$ such that whenever
    \[
    \inf\left\{\|\vec w_0-\vec\psi\|_\x:\vec\psi\in S\right\}<\delta,
    \]
the solution $\vec w$ of the system \eqref{E:B6_system} with $\vec w(0)=\vec w_0$ exists for all $t>0$ and satisfies
    \[
    \sup_{t>0}\inf\left\{\|\vec w(t)-\vec\psi\|_\x:\vec\psi\in S\right\}<\epsilon.
    \]
Otherwise we say the set $S$ is $\x$-unstable.
\end{definition}

In this section we show that the stability of the set of ground states is determined by the convexity of the function
    \begin{equation}\label{E:d_definition}
    d(c)=E(\vec\ff)+cQ(\vec \ff)
    \end{equation}
where $\vec\ff=(\ff,-c\ff)$ and $\ff\in\g$.

\begin{theorem}\label{T:stability} Denote $\vg=\{\vec\ff=(\ff,-c\ff):\ff\in\g\}$.
  Suppose $c^2<1$ and $\beta<\beta_*=2\sqrt{1-c^2}$. If $d''(c)>0$ then $\vg$ is $\x$-stable.
\end{theorem}

Before proving Theorem \ref{T:stability}, we state the basic properties of the function $d$. We first note that, for any $\vec w=(u,v)\in\x$ we have
\begin{equation}\label{E:EQIK_relation}
  E(\vec w)+cQ(\vec w)=\frac12I(u)-\frac1{p+1}K(u)+\frac12\int_\rr(cu+v)^2\dd x.
\end{equation}
Applying this to $\vec\ff=(\ff,-c\ff)$ where $\ff\in\g$ and using the fact that $I(\ff)=K(\ff)$, we have
    \[
    E(\vec\ff)+cQ(\vec\ff)=\frac12I(\ff)-\frac1{p+1}K(\ff)=\frac{p-1}{2(p+1)}I(\ff).
    \]
By relation \eqref{E:ground_state_definition} this implies that
\begin{equation}\label{E:dm_relation}
  d(c)=\frac{p-1}{2(p+1)}m(\beta,c)^{\frac{p+1}{p-1}}
\end{equation}
so $d$ is well-defined, and the properties of $d$ may be deduced by studying the properties of the function $m(\beta,c)$. By reasoning similar to that in \cite{levandosky2} we obtain the following.

\begin{lemma}\label{L:d_properties}
On the domain $D=\{(\beta,c):c^2<1, \beta<2\sqrt{1-c^2}\}$, $d$ is continuous and strictly decreasing in both $|c|$ and $\beta$. For each fixed $\beta$, $d_c(\beta,c)$ exists for all but countably many $c$, and for fixed $c$, $d_\beta(\beta,c)$ exists for all but countably many $\beta$. At points of differentiability we have
\begin{align*}
  d_c(\beta,c)&=Q(\vec\ff)=-c\int\ff^2\dd x\\
  d_\beta(\beta,c)&=-\frac12\int\ff_x^2\dd x
\end{align*}
for any $\ff\in\g$.
\end{lemma}

For the remainder of this section we fix $\beta<2$ and regard $d$ as a function of $c$ only. We denote by
    \[
    U_\epsilon\equiv U_{\beta,c;\epsilon}=\left\{\vec w\in\x\mid \inf_{\varphi\in \g}\|\vec w-\vec\varphi\|_{\x}<\epsilon\right\}
    \]
the $\epsilon$-neighborhood of the set of ground states $\g$.

\begin{lemma}\label{L:c_map} For each $(\beta,c)\in D^+=\{(\beta,c):0\leq c<1,\beta<2\sqrt{1-c^2}\}$, there exists $\epsilon>0$ such that the mapping $c:U_\epsilon\to\mathbb R$ defined by
    \[
    c(\vec w)=c(u,v)=d^{-1}\left(\frac{p-1}{2(p+1)}K(u)\right)
    \]
is continuous.
\end{lemma}

\proof Since $d$ is monotone decreasing and continuous, it follows that for fixed $\beta<2$ its inverse with respect to $d$, $d^{-1}$, is defined and continuous in some $\delta$-neighborhood of $d(c)$. It therefore remains to show that $\frac{p-1}{2(p+1)}K(u)$ lies in this neighborhood when $u\in U_\epsilon$ and $\epsilon$ is sufficiently small. First observe that for any $u_1,u_2\in H^2(\rr)$ we have
    \begin{align*}
    |K(u_1)-K(u_2)|&\leq(p+1)\int_{\rr} |F(u_1)-F(u_2)|\dd x\\
    &=(p+1)\int_\rr |f(\mu(x)u_1+(1-\mu(x))u_2)||u_1-u_2|\dd x\\
    &= (p+1)\int_\rr C|\mu(x)u_1+(1-\mu(x))u_2)|^p|u_1-u_2|\dd x\\
    &\leq C(\|u_1\|_{L^{p+1}}+\|u_2\|_{L^{p+1}})^p\|u_1-u_2\|_{L^{p+1}}.
    \end{align*}
Thus by the embedding of $H^2(\rr)$ into $L^{p+1}(\rr)$, it follows that $K$ is locally Lipschitz on $H^2(\rr)$. Given any $\varphi\in\g$ the coercivity condition \eqref{E:I_coercive} and relations \eqref{E:dm_relation} and \eqref{E:ground_state_definition} imply that
    \[
    \|\ff\|_{H^2(\rr)}\leq {C}^{-1}I(\ff)=C^{-1}\frac{2(p+1)}{p-1}d(c).
    \]
Hence the set of ground states $\g$ is a bounded subset of $\x$. Consequently the neighborhood $U_{\epsilon}$ is bounded for any $\epsilon>0$. Thus since $\frac{p-1}{2(p+1)}K(\varphi)=d(c)$ for any $\ff\in\g$, the Lipschitz continuity of $K$ and boundedness of $U_\epsilon$ imply that $\frac{p-1}{2(p+1)}K(u)$ lies in the $\delta$-neighborhood of $d(c)$ for all $\vec w=(u,v)\in U_\epsilon$ if $\epsilon>0$ is small enough.
\fim

\begin{lemma}\label{L:c_bounds} Suppose $d''(c)>0$. Then there exists some $\epsilon_c>0$ such that for any $\\ff\in\g$ and any $\vec w\in U_{\epsilon}$ we have
    \[
    E(\vec w)-E(\vec\ff)+c(\vec w)(Q(\vec w)-Q(\vec\ff))\geq\frac14d''(c)(c(\vec w)-c)^2.
    \]
\end{lemma}

\proof  Using Taylor's Theorem and the fact that $d'(c)=Q(\vec\ff)$ we have
    \[
    d(c_1)=d(c)+Q(\vec\ff)(c_1-c)+\frac12d''(c)(c_1-c)^2+o(|c_1-c|)^2
    \]
for $c_1$ near $c$. Thus for $c_1$ is some $\delta$-neighborhood of $c$ we have
    \[
    d(c_1)\geq d(c)+Q(\vec\ff)(c_1-c)+\frac14d''(c)(c_1-c)^2.
    \]
By Lemma \ref{L:c_map} it then follows that for sufficiently small $\epsilon_c$ and $\vec w=(u,v)\in U_{\epsilon_c}$ we have
    \begin{align*}
    d(c(\vec w))&\geq d(c)+Q(\vec\ff)(c(\vec w)-c)+\frac14d''(c)(c(\vec w)-c)^2\\
    &=E(\vec\ff)+cQ(\vec\ff)+Q(\vec\ff)(c(\vec w)-c)+\frac14d''(c)(c(\vec w)-c)^2\\
    &=E(\vec\ff)+c(\vec w)Q(\vec\ff)+\frac14d''(c)(c(\vec w)-c)^2.
    \end{align*}
Next suppose $\psi\in\mathscr{G}(\beta,c(\vec w))$. Then
$I(\psi;\beta,c(\vec w))=K(\psi)=\frac{2(p+1)}{p-1}d(c(\vec w))=K(u)$ and $\psi$ minimizes
$I(\cdot;\beta,c(\vec w))$ subject to the constraint $K(\cdot)=K(u)$. By \eqref{E:EQIK_relation} we have
    \begin{align*}
    E(\vec w)+c(\vec w)Q(\vec w)&=\frac12I(u;\beta,c(\vec w))-\frac1{p+1}K(u)+\frac12\int_\rr(cu+v)^2\dd x\\
    &\geq \frac12I(\psi;\beta,c(\vec w))-\frac1{p+1}K(\psi)\\
    &=d(c(\vec w)).
    \end{align*}
Combining these inequalities proves the desired result.
\fim

\noindent{\it Proof of Theorem \ref{T:stability}}.
Suppose $\g$ is $\x$-unstable, and choose initial data $\vec w_0^k$ such that
    \[
    \inf_{\ff\in\g}\|\vec w_0^k-\vec\ff\|_{\x}<\frac1k.
    \]
This implies that there exist $\ff_k\in \g$ such that
    \begin{equation}\label{E:stability1}
    \lim_{k\to\infty}\|\vec w_0^k-\vec\ff_k\|_\x=0.
    \end{equation}
Denote by $\vec w^k(t)$ the solutions of \eqref{E:B6_system} with $\vec w^k(0)=\vec w_0^k$. Then there exist some $\delta>0$ and times $t_k$ (for each $k>\frac1\delta$) such that
    \[
    \inf_{\ff\in\g}\|\vec w^k(t_k)-\vec\ff\|_{\x}=\delta.
    \]
Without loss of generality we may also suppose that $\delta<\epsilon_c$ and therefore $\vec w^k(t_k)\in U_{\epsilon_c}$, so that Lemma \ref{L:c_bounds} implies
\begin{equation}\label{E:stability2}
    E(\vec w^k(t^k))-E(\vec\ff)+c(\vec w^k(t_k)))(Q(\vec w^k(t_k))-Q(\vec\ff))\geq\frac14d''(c)[c(\vec w^k(t_k))-c]^2.
\end{equation}
 Next, using the fact that $E$ and $Q$ are continuous on $\x$ and conserved for solutions of equation \eqref{E:B6_system}, we have from equation \eqref{E:stability1} that
    \begin{equation}
    \lim_{k\to\infty} E(\vec w^k(t_k))-E(\vec\ff_k)=\lim_{k\to\infty}E(\vec w_0^k)-E(\vec\ff_k)=0
    \end{equation}
and
    \begin{equation}
    \lim_{k\to\infty} Q(\vec w^k(t_k))-Q(\varphi_k)=\lim_{k\to\infty}Q(\vec w_0^k)-Q(\vec\ff_k)=0.
    \end{equation}
By Lemma \ref{L:c_map}, the sequence of scalars $c(\vec w^k(t_k)))$ is bounded, and thus equation \eqref{E:stability2} implies that
    \[
    \lim_{k\to\infty} c(\vec w^k(t_k))=c.
    \]
By continuity of $d$, this implies that $K(u^k(t_k))=\frac{2(p+1)}{p-1}d(c(\vec w^k(t_k)))$ converges to $\frac{2(p+1)}{p-1}d(c)$.
Using the relation \eqref{E:EQIK_relation} and the fact that $d(c)=E(\vec\ff_k)+cQ(\vec\ff_k)$, it follows that
    \begin{align*}
    \frac12I(u^k(t_k))&=E(\vec w^k(t_k))+cQ(\vec w^k(t_k))+\frac1{p+1}K(u^k(t_k))-\frac12\int\rr(cu^k(t_k)+v^k(t_k))^2\\
    &\leq E(\vec w^k(t_k))-E(\vec\ff_k)+c(Q(\vec w^k(t_k))-Q(\vec\ff_k))+\frac1{p+1}K(u^k(t_k))+d(c)\\
    &\to\frac{p+1}{p-1}d(c).
    \end{align*}
Thus
    \[
    \limsup_{k\to\infty} I(u^k(t_k))\leq\frac{2(p+1)}{p-1}d(c),
    \]
 which implies that $u^k(t_k)$ is a minimizing sequence for $\lambda=\frac{2(p+1)}{p-1}d(c)$. By Theorem \ref{T:existence}, there is a translated subsequence, renamed $u^k(t_k)$, that converges in $H^2(\rr)$ to some $\ff\in\g$. To control the second component of $\vec w^k(t_k)$ observe that
\begin{align*}
 \frac12\int\rr(cu^k(t_k)+v^k(t_k))^2&=-\frac12I(u^k(t_k))+\frac1{p+1}K(u^k(t_k))
 +E(\vec w^k(t_k))+cQ(\vec w^k(t_k))\\
 &\to -d(c)+d(c)=0.
\end{align*}
Hence $v^k(t_k)$ converges in $L^2(\rr)$ to $-c\ff$, and thus $\vec w^k(t_k)$ converges in $\x$ to $\vec\ff=(\ff,-c\ff)$.
Therefore
    \[
    \inf_{\ff\in\g}\|\vec w^k(t_k)-\vec\ff\|_{\x}=0,
    \]
a contradiction. This completes the proof of Theorem \ref{T:stability}.
\fim

\section{Instability}\label{S:instability}
In this section we establish conditions that imply orbital instability of solitary waves.

The following theorem is a key point in the proof of the instability.
\begin{theorem}\label{time-bound-theorem}
Let $\LL^2u_0$ and $\LL^2 v_0$ be in $L^1(\rr)$. Assume that $|f(s)|=O(|s|^p)$ and $|f'(s)|=O(|s|^{p-1})$, as $s\to\infty$, for $p>1$. Suppose also that $\vu=(u,v)$ is a solutions of \eqref{E:B6_system} with $\vu(0)=\vdu$. Then
\begin{enumerate}[(i)]
\item if $p\geq 2$,
\[
\sup_{x\in\rr}\left|\int_{-\infty}^x\vu(z,t)\dd z\right|\leq C\left(1+t^{2/3}+t^{4/5}\right),
\]
\item if $1<p<2$,
\[
\sup_{x\in\rr}\left|\int_{-\infty}^x\vu(z,t)\dd z\right|\leq C\left(t^{1-(p-1)/3}+t^{1-(p-1)/5}\right),
\]
\end{enumerate}
for $t\in(0,T)$, where $T$ is the maximum existence time for $\vu$, and the constant $C>0$  depends only on $\|\vdu\|_\x$, $f$ and $\sup_{t\in[0,T)}\|\vu(t)\|_\x$.
\end{theorem}

To prove Theorem \ref{time-bound-theorem}, a series of useful lemmas are laid out. The first one is the well-known Van der Corput lemma \cite{stein} as follows.
\begin{lemma}\label{van-der-corput}
Let $h$ be either convex or concave on $[a,b]$ with $-\infty\leq a<b\leq+\infty$. Then
\[
\left|\int_a^b\ee^{\ii h(\xi)}\dd\xi\right|\leq 4\left(\min_{\xi\in[a,b]}|h''(\xi)|\right)^{-1/2},
\]
if $h''\neq0$ in $[a,b]$.
\end{lemma}

\begin{lemma}\label{L:van_der_corput_corollary}
Suppose $h$ is twice differentiable on $\mathbf R$ and
\begin{enumerate}[(i)]
  \item $h''$ has finitely many zeroes, all of which are of order $q_1$ or less.
  \item there exist positive constants $C_1$ and $C_2$ such that $|h''(\xi)|\geq C_2|\xi|^{q_2}$ whenever $|\xi|\geq C_1$.
\end{enumerate}
Then there exists a constant $C$ such that
    \[
    \left|\int_\rr\ee^{\ii th(\xi)}\dd\xi\right|\leq Ct^{-1/(2+q_2)}
    \]
for $0<t<1$, and
    \[
    \left|\int_\rr\ee^{\ii th(\xi)}\dd\xi\right|\leq Ct^{-1/(2+q_1)}
    \]
for $t\geq1$.
\end{lemma}

\begin{proof} First suppose $0<t<1$. Given $R>C_1$, set
$\Omega_1=\{\xi:|\xi|<R\}$ and $\Omega_2=\{\xi:|\xi|\geq R\}$. Then $|h''(\xi)|\geq C_2R^{q_2}$ for $\xi\in\Omega_2$, so by Lemma \ref{van-der-corput} we have
    \[
    \left|\int_{\Omega_2}\ee^{\ii th(\xi)}\dd\xi\right|\leq CR^{-q_2/2}t^{-1/2},
    \]
while on $\Omega_1$ we have
    \[
    \left|\int_{\Omega_1}\ee^{\ii th(\xi)}\dd\xi\right|\leq R.
    \]
For $0<t<1$ sufficiently small, we may set $R=t^{-1/(2+q_2)}$ and the result follows.

Next suppose $t\geq 1$and let $\xi_1,\ldots,\xi_n$ denote the zeroes of $h''$. For $\epsilon>0$ let $I_k=\{\xi:|\xi-\xi_k|<\epsilon\}$ for each $k$, and set
$\Omega_{1}=\bigcup_{k}I_k$ and $\Omega_{2}=\rr\backslash\Omega_{11}$. Then we have
    \[
    \left|\int_{\Omega_{1}}\ee^{\ii th(\xi)}\dd\xi\right|\leq C\epsilon.
    \]
Since each zero of $h''$ is at most order $q_1$, there exists $C_3>0$ such that for $\epsilon>0$ sufficiently small, we have
$|h''(\xi)|\geq C_3\epsilon^{q_1}$ for $\xi\in\Omega_{2}$. It then follows from Lemma \ref{van-der-corput} that
    \[
    \left|\int_{\Omega_{2}}\ee^{\ii th(\xi)}\dd\xi\right|\leq Ct^{-1/2}\epsilon^{-q_1/2}.
    \]
For $t$ sufficiently large, we may set $\epsilon=t^{-1/(2+q_1)}$ and the estimate follows.
\end{proof}
\fim

\begin{lemma}\label{oscillatory-integral}
Let $\beta<2$ and set $h(\xi,\al)=\sqrt{\xi^2-\beta\xi^4+\xi^6}+\al\xi$.
\begin{enumerate}[(i)]
  \item If $\beta\neq0$, there exists a positive constant $C$ such that
\[
\sup_{\al\in\rr}\left|\int_\rr\ee^{\ii t h(\xi,\al)}\dd\xi\right|\leq Ct^{-1/3}
\]
for all $t>0$.
\item If $\beta=0$ there exists a positive constant $C$ such that
\[
\sup_{\al\in\rr}\left|\int_\rr\ee^{\ii t h(\xi,\al)}\dd\xi\right|\leq C(t^{-1/3}+t^{-1/5})
\]
for all $t>0$.
\end{enumerate}
\end{lemma}

\proof First observe that $h$ is an even $C^\infty$-function in $\rr\setminus\{0\}$ with
\begin{equation}\label{h2-1}
\frac{\partial^2 h}{\partial\xi^2}=\frac{|\xi|(-3\beta+(2\beta^2+10)\xi^2-9\beta\xi^4+6\xi^6)}{(1-\beta\xi^2+\xi^4)^{3/2}}.
\end{equation}
Since the polynomial $-3\beta+(2\beta^2+10)\xi^2-9\beta\xi^4+6\xi^6$ is increasing in $\xi^2$ for $\beta<2$, it follows that
\begin{enumerate}[(i)]
  \item if $\beta>0$, $h''$ has three simple zeroes, $0$, $\xi_0>0$ and $-\xi_0$,
  \item if $\beta<0$, $h''$ has one simple zero, $\xi=0$,
  \item if $\beta=0$, $h''$ has a zero of order 3 at $\xi=0$.
\end{enumerate}
In cases (i) and (ii) the result then follows from Lemma \ref{L:van_der_corput_corollary} with $q_1=q_2=1$, while for $\beta=0$ it follows from the same lemma with $q_1=3$ and $q_2=1$.
\fim

\begin{lemma}\label{Lp-bessel}
If $u\in L^p(\rr)$, $1\leq p\leq \infty$, then $\Lambda^{-2}u\in L^p(\rr)$ and $\|\Lambda^{-2}u\|_{L^p(\rr)}\leq C\|u\|_{L^p(\rr)}$, for some $C>0$.
\end{lemma}
\proof
The proof follows from Young's inequality and the fact $G(x)=\exp(-|x|)\in L^1(\rr)$, where $\Lambda^{-2}u=G\ast u$  and $\what{G}(\xi)=(1+\xi^2)^{-1}$.
\fim
The following lemma gives a time estimate on the solutions of the linearized problem.
\begin{lemma}\label{est-lin-oper}
Let $S(t)$ be the $C_0$ group of unitary operators for the linearized  problem of \eqref{E:B6_system}
\[
\vu_t+
\left(\begin{array}{lr}
0&-1\\
-1-\beta\partial_x^2-\partial_x^4&0
\end{array}\right)
\vu_x=0,
\]
with $\vu(0)=\vdu(u_0,v_0)$. If $\Lambda^2u_0\in L^1(\rr)$ and $v_0\in L^1(\rr)$, then $S(t)\vdu\in L^\infty(\rr)\times L^\infty(\rr)$ and
\[
\|S(t)\vdu\|_{L^\infty(\rr)\times L^\infty(\rr)}
\leq
C\left(t^{-1/3}+t^{-1/5}\right)\left(\|\Lambda^2u_0\|_{L^1(\rr)}+\|v_0\|_{L^1(\rr)}\right),
\]
where $C>0$ is a constant.
\end{lemma}

\proof
Since
\[
\vu(t)=S(t)\vdu(x)=\int_\rr\ee^{\ii x\xi}
\left(\begin{array}{lr}
\cos(t\xi\vtt(\xi))&\frac{\ii}{\vtt(\xi)}\sin(t\xi\vtt(\xi))\\ \\
\ii\vtt(\xi)\sin(t\xi\vtt(\xi))&\cos(t\xi\vtt(\xi))
\end{array}\right)\what{\vdu}(\xi)\dd\xi,
\]
where $\vtt(\xi)=\sqrt{1-\beta\xi^2+\xi^4}$. It is deduced from Fubini's theorem and Lemmas \ref{oscillatory-integral} and \ref{Lp-bessel} that
\[
\begin{split}
|\vu(t)|&=|S(t)\vdu(x)|
\lesssim
\sum\left|\int_\rr\left(\what{u_0}\pm\frac{1}{\vtt(\xi)} \widehat{{v_0}}\right)\ee^{\ii t\xi(\vtt(\xi)\pm x/t)}\dd\xi\right|
+
\sum\left|\int_\rr\left(\what{v_0}\pm\vtt(\xi) \widehat{{u_0}}\right)\ee^{\ii t\xi(\vtt(\xi)\pm x/t)}\dd\xi\right|
\\
&
\lesssim
\sum\left|\int_\rr\left(\what{u_0}\pm \widehat{\Lambda^{-2}{v_0}}\right)\ee^{\ii t\xi(\vtt(\xi)\pm x/t)}\dd\xi\right|
+
\sum\left|\int_\rr\left(\what{v_0}\pm \widehat{\Lambda^{-2}{u_0}}\right)\ee^{\ii t\xi(\vtt(\xi)\pm x/t)}\dd\xi\right|
\\
&
\lesssim
\sum\int_\rr\left|u_0(z)\pm\Lambda^{-2}v_0(z)\right|\left|\int_\rr\ee^{\ii t\xi(\vtt(\xi)\pm x/t-z/t)}\dd\xi\right|\dd z\\
&\para
+
\sum\int_\rr\left|v_0(z)\pm\Lambda^{-2}u_0(z)\right|\left|\int_\rr\ee^{\ii t\xi(\vtt(\xi)\pm x/t-z/t)}\dd\xi\right|\dd z,
\end{split}\]
where the sums are over all two sign combinations. Therefore, we obtain from Lemma \ref{oscillatory-integral} that
\[
\begin{split}
|\vu(t)|&
\lesssim
\sup_{\al\in\rr}\left|\int_\rr\ee^{\ii th(\xi,\al)}\dd\xi\right|\left(\|\vdu\|_{L^1(\rr)\times L^1(\rr)}+\|\Lambda^{-2}v_0\|_{L^1(\rr)}+\|\Lambda^{2}u_0\|_{L^1(\rr)}\right)
\\
&
\lesssim \left(t^{-1/3}+t^{-1/5}\right)\left(\|v_0\|_{L^1(\rr)}+\|\Lambda^{2}u_0\|_{L^1(\rr)}\right).
\end{split}
\]
for $t>0$.
Hence, for some $C>0$, it is concluded
\[
\begin{split}
|\vu(t)|&
\leq C
\left(t^{-1/3}+t^{-1/5}\right)\left(\|v_0\|_{L^1(\rr)}+\|\Lambda^{2}u_0\|_{L^1(\rr)}\right).
\end{split}
\]
\fim

\indent A proof of Theorem \ref{time-bound-theorem} is now in sight.\\

\noindent\textbf{Proof of Theorem \ref{time-bound-theorem}.}\q
Let $\vec w(t)=S(t)\vdu$, then $\vec w$ satisfies
\begin{equation}\label{integral-form-0}
\vec w_t+\left(\begin{array}{lr}
0&-1\\
-1-\beta\partial_x^2-\partial_x^4&0
\end{array}\right)\vec w_x=0,\para\mbox{with}\q\vec w(0)=\vdu.
\end{equation}
Then the solution $\vu(t)$ of \eqref{E:B6_system} can be written
\[
\vu(t)=\vec w(t)-\partial_x\int_0^t S(t-\tau)\left(\begin{array}{c}
0\\f(u(\tau))
\end{array}\right)\dd\tau.
\]
We should estimate
\begin{equation}\label{integral-form-1}
\vec U(t)=\vec W(t)-\int_0^tS(t-\tau)\left(\begin{array}{c}
0\\f(u(\tau))
\end{array}\right)\dd\tau,
\end{equation}
where
\[
\vec U(x,t)=\int_{-\infty}^x\vu(z,t)\dd z\q\mbox{and}\q\vec W(x,t)=\int_{-\infty}^x\vec w(z,t)\dd z.
\]
First observe using \eqref{integral-form-0} that
\[
\vec W(t)=\vec U_0-\int_0^t S(\tau)
\left(\begin{array}{lr}
0&-1\\
-1-\beta\partial_x^2-\partial_x^4&0
\end{array}\right)
\vdu\dd\tau,
\]
where $\vec U_0(x)=\int_{-\infty}^x\vdu(z)\dd z$. Now Lemma \ref{est-lin-oper} implies that
\[\begin{split}
|\vec W(x,t)|&\lesssim
\|\vdu\|_{L^1(\rr)\times L^1(\rr)}
+
\left(\|\Lambda^2 u_0\|_{L^1(\rr)}+\|\Lambda^2 v_0\|_{L^1(\rr)}\right)\int_0^t\left(\tau^{-1/3}+\tau^{-1/5}\right)\dd\tau
\\
&
\lesssim
\left(\|\Lambda^2 u_0\|_{L^1(\rr)}+\|\Lambda^2 v_0\|_{L^1(\rr)}\right)
\left(1+t^{2/3}+t^{4/5}\right).
\end{split}\]
Setting
\[
\vec Y(x,t)=\int_0^tS(t-\tau)\left(\begin{array}{c}
0\\f(u(\tau))
\end{array}\right)\dd\tau,
\]
and using Lemma \ref{est-lin-oper} again, it follows that
\begin{equation}\label{est-1}
|\vec Y(x,t)|\lesssim \int_0^t\left((t-\tau)^{-1/3}+(t-\tau)^{-1/5}\right)\|f(u(\tau))\|_{L^1(\rr)}\dd\tau.
\end{equation}
On the other hand, it is deduced from Cauchy-Schwarz inequality that
\begin{equation}\label{est-2}
|\vec Y(x,t)|\lesssim \int_0^t
\int_\rr\left(1+\frac{1}{\sqrt{1-\beta\xi^2+\xi^6}}\right)|f(\what{u(\xi,\tau)})|\dd\xi
\dd\tau
\lesssim
\int_0^t\|f(u(\cdot,\tau))\|_{H^2(\rr)}\dd\tau.
\end{equation}
Since $H^2(\rr)\hookrightarrow L^\infty(\rr)$ and $|f(s)|=O(|s|^p)$ and $|f'(s)=O(|s|^{p-1})|$ as $s\to0$, for $p>1$, it transpires that $\|f(u)\|_{L^1(\rr)}\leq C$, provided $p\geq2$, for some positive constant $C$ which depends only on $f$ and $\sup_{t\in[0,T)}\|\vu(t)\|_\x$. Hence, if $p\geq2$,
\[
|\vec Y(x,t)|\lesssim
\int_0^t
\left((t-\tau)^{-1/3}+(t-\tau)^{-1/5}\right)\dd\tau
\leq C\left(t^{2/3}+t^{4/5}\right).
\]

If $1<p<2$, it is straightforward to check that $\|f(u)\|_{H^{1,2/p}(\rr)}\leq C$, for some $C>0$. Since \eqref{est-1} and \eqref{est-2} hold for any $f\in L^1(\rr)\cap H^2(\rr)$, a straightforward interpolation thus can be applied for the mapping $S(t-\tau)$ as in \eqref{est-1} and \eqref{est-2}. Thus the same argument proves that
\[\begin{split}
|\vec Y(x,t)|&\lesssim
\int_0^t\left((t-\tau)^{-1/3}+(t-\tau)^{-1/5}\right)^{p-1}\|f(u(\cdot,\tau))\|_{H^{1,2/p}(\rr)}\dd\tau\\
&
\leq C
\left(t^{1-(p-1)/3}+t^{1-(p-1)/5}\right).
\end{split}\]
By combining the estimates of $\vec Y$ and $\vec W$, the proof of Theorem \ref{time-bound-theorem} is now completed.
\fim

Given $\vff\in\g$ and $\epsilon>0$, we define the ``tube''
\[
\Omega_{\vff,\eps}=\left\{u\in\x;\;\inf_{v\in\mathcal{O}_{\vff} }\|v-u\|_\x<\eps\right\}
\]
and the operator
\[
H=L''(\vff)=E''(\vff)+cQ''(\vff).
\]
The main instability result is the following.

\begin{theorem}\label{T:instability}
Suppose $c^2<1$ and $\beta<\beta_*$. If there exists $\vec\psi\in\x$ such that $\partial_x\vp\in L^1(\rr)\times L^1(\rr)$ and $\vp\in H^2(\rr)\times H^2(\rr)$
\begin{enumerate}
  \item $\left\langle Q'(\vec\ff),\partial_x\vec\psi\right\rangle=0$,
  \item $\left\langle H\partial_x\vec\psi,\partial_x\vec\psi\right\rangle<0$,
\end{enumerate}
  then $\mathcal{O}_{\vff}$ is $\x$-unstable.
\end{theorem}

\begin{lemma}\label{implicit}
Let  $c^2<1$ and $\beta<\beta_\ast$ and $\vff\in\g$ be fixed. There exist $\eps_0>0$ and a unique $C^2$ map $\al:\Omega_{\vff,\eps_0}\to\rr$ such that $\al(\vff)=0$, and for all $\vu\in \Omega_{\vff,\eps_0}$ and any $r\in\rr$,
\begin{enumerate}[(i)]
\item $\<\vu(\cdot-\al(\vu)),\partial_x\vff\>=0$,
\item $\al(\vu(\cdot+r))=\al(\vu)-r$,
\item $\al'(\vu)=\frac{1}{{\left\<\vu,\partial_x^2\vff(\cdot-\al(\vu))\right\>}}\partial_x\vff(\cdot-\al(\vu))$, and
\item $\<\al'(\vu),\vu\>=0$, if $\vu\in\mathcal{O}_{\vff}$.
\end{enumerate}
\end{lemma}
\proof
The proof follows the line of reasoning laid down in Theorem 3.1 in \cite{ribeiro} and Lemma 3.8 in \cite{liutom}.
\fim

Let $\vp$ be as in  Theorem \ref{T:instability}. Define another vector field $\bvp$ by
\[
\bvp(\vu)=\kk\partial_x\vp(\cdot-\al(\vu))-
\frac
{\left\<\vu,\partial_x\vp(\cdot-\al(\vu))\right\>}
{\left\<\vu,{\partial_x^2}{\vff(\cdot-\al(\vu))}\right\>}\kk\partial_x^2\vff(\cdot-\al(\vu)),
\]
for $\vu\in\Omega_{\vff,\eps}$, where
$\kk=\left(\begin{array}{ll}
0&1\\
1&0
\end{array}\right)$.
Geometrically, $\bvp$  can be interpreted as the derivative of the orthogonal component of $\tau_{-\al(\cdot)}\vp$ with regard to $\tau_{-\al(\cdot)}\partial_x\vff$.

\begin{lemma}\label{vector-field}
Let $\vp$ be as in  Theorem \ref{T:instability}. Then the map $\bvp:\Omega_{\vff,\eps_0}\to\x$ is $C^1$ with bounded derivative. Moreover,
\begin{enumerate}[(i)]
\item $\bvp$ commutes with translations,
\item $\<\bvp(\vu),\kk\vu\>=0$, if $\vu\in \Omega_{\vff,\eps_0}$,
\item $\bvp(\vff)=\partial_x\kk\vp$, if $\<\vff,\partial_x\vp\>=0$.
\end{enumerate}
\end{lemma}

\proof
The proof follows the same lines from the proof of Lemma 3.5 in \cite{angulo-1} or Lemma 3.3 in \cite{angulo-2}.
\fim
Before starting with the proof of Theorem \ref{T:instability}, we state and prove the following lemma which shows the boundedness of the Liapunov function (see \eqref{Liapunov}).
\begin{lemma}\label{time-estimate}
Let $\vp$ be as in Theorem \ref{T:instability}, $\vdu=(u_0,v_0)$ be in $\Omega_{\vff,\eps_3}$ such that $\Lambda u_0,\Lambda v_0\in L^1(\rr)$ and $f$ satisfy the assumptions of Theorem \ref{time-bound-theorem}. If $\vu(t)$ is a solution of \eqref{E:B6_system} which corresponds to the initial data $\vdu$ and $\vu(t)\in\Omega_{\vff,\eps_3}$, for $t\in [0,T]$, then
\begin{equation}\label{L:Estimate}
\left|\int_\rr\vp(x-\al(\vu(t)))\cdot\vu(x,t)\dd x\right|
\leq C\left(1+t^{2/3}+t^{4/5}\right)
\end{equation}
for $t\in[0,T)$, where $T$ is the maximum existence time for $\vu$,  and the constant $C>0$  depends on $\|\Lambda^2 u_0\|_{L^1(\rr)}$, $\|\Lambda^2 v_0\|_{L^1(\rr)}$, $f$ and $\vff$.
\end{lemma}

\proof
Let  $\hb$ be the Heaviside function and $\vec\ga=\int_\rr\partial_x\vp(x)\dd x$. Then the left hand side of \eqref{L:Estimate} may be written
\[
\int_\rr\left(\vp\left(x-\al(\vu(t))\right)-\vec\ga\hb(x-\al(\vu(t)))\right)\cdot\vu(x,t)\dd x
+
\vec\ga\cdot\int_{\al(\vu(t))}^{+\infty}\vu(x,t)\dd x.
\]
So it follows from Cauchy-Schwarz inequality and Theorem \ref{time-bound-theorem} that
\[\begin{split}
&\left|\int_\rr\vp(x-\al(\vu(t)))\cdot\vu(x,t)\dd x\right|\\
&\q\q\leq
\left\|\vp-\vec\ga\hb\right\|_{\lt\times\lt}\left\|\vu(t)\right\|_{\lt\times\lt}
+
C\left(1+t^{2/3}+t^{4/5}\right).
\end{split}\]
We show that $\vp-\vec\ga\hb\in\lt\times\lt$. Indeed, Minkowski's inequality yields that
\[
\begin{split}
\left\|\vp-\vec\ga\hb\right\|_{\lt\times\lt}&
\leq
\left(\int_{-\infty}^0|\vp(x)|^2\dd x\right)^{1/2}
+
\left(\int^{+\infty}_0|\vp(x)-\vec\ga\hb(x)|^2\dd x\right)^{1/2}\\
&
=\left(\int_{-\infty}^0|\vp(x)|^2\dd x\right)^{1/2}
+
\left(\int^{+\infty}_0\left|\int_x^{+\infty}\partial_x\vp(z)\dd z\right|^2\dd x\right)^{1/2}\\
&
\leq\|\vp\|_{\lt\times\lt}
+
\int_\rr|\partial_x\vp(x)|\sqrt{|x|}\dd x\\
&\leq\|\vp\|_{H^2(\rr)\times H^2(\rr)}<+\infty.
\end{split}
\]
Hence, for $t\in[0,T)$, we obtain
\[
\left|\int_\rr\vp(x-\al(\vu(t)))\cdot\vu(x,t)\dd x\right|
\leq C\left(1+t^{2/3}+t^{4/5}\right).
\]
\fim

All the elements are now in place to prove the instability result in
Theorem \ref{T:instability}.

\noindent\textbf{Proof of Theorem \ref{T:instability}.}\q
First we claim that there exist $\eps_3>0$ and $\si_3>0$ such that for each $\vdu\in\Omega_{\vff,\eps_3}$,
\begin{equation}\label{taylor-exp-1}
L(\vff)\leq L(\vdu)+\p(\vdu)s,
\end{equation}
for some $s\in(-\si_3,\si_3)$, where $\p(\vu)=\<L'(\vu),\bvp(\vu)\>$.

For $\vdu\in\Omega_{\vff,\eps_0}$, where $\eps_0$ is given in Lemma \ref{implicit}, consider the initial value problem
\begin{equation}\label{initial-value}
\begin{array}{ll}
\frac{{\rm d} }{{\rm d} s}\vu(s)=\bvp(\vu(s))\\
\vu(0)=\vdu.
\end{array}
\end{equation}
By Lemma \ref{vector-field}, we have that \eqref{initial-value} admits for each $\vdu\in\Omega_{\vff,\eps_0}$ a unique maximal solution $\vu\in C^2((-\si,\si);\Omega_{\vff,\eps_0})$, where $\si\in(0,+\infty]$. Moreover for each $\eps_1<\eps_0$, there exists $\si_1>0$ such that $\si(\vdu)\geq\si_1$, for all $\vdu\in \Omega_{\vff,\eps_1}$. Hence we can define for fixed $\eps_1$, $\si_1$, the following dynamical system
\[
\begin{array}{cc}
\uu:(-\si_1,\si)\times \Omega_{\vff,\eps_1}\longrightarrow\Omega_{\vff,\eps_0}\\
\q\q(s,\vdu)\mapsto\uu(s)\vdu,
\end{array}\]
where $s\to\uu(s)\vdu$ is the maximal solution of \eqref{initial-value} with initial data $\vdu$. It is also clear from Lemma \ref{vector-field} that $\uu$ is a $C^1-$function, also we have that for each $\vdu\in \Omega_{\ff,\eps_1}$, the function $s\to\uu(s)\vdu$ is $C^2$ for $s\in(-\si_1,\si_1)$, and the flow $s\to\uu(s)\vdu$ commutes with translations. One can also observe from the relation
\[
\uu(t)\vff=\vff+\int_0^t\tau_{\al(\uu(s)\vff)}\partial_x\vp\dd s
-\int_0^t\rho(s)\tau_{\al(\uu(s)\vff)}\partial_x^2\vff\dd s
\]
that $\uu(s)\vff\in H^r(\rr)$, $r>3/2$, for all $s\in(-\si_1,\si_1)$, where
\[
\rho(s)=\frac{\left\<\uu(s)\vff,\tau_{\al(\uu(t)\vff)}\partial_x\vp\right\>}
{\left\<\uu(t)\vff,\tau_{\al(\uu(t)\vff)}\partial_x^2\vff\right\>}.
\]
Now we obtain from Taylor's theorem that there is $\varrho\in(0,1)$ such that
\[
L(\uu(s)\vdu)=L(\vdu)+\p(\vdu)s+\frac{1}{2}R(\uu(\varrho s)\vdu)s^2,
\]
where $R(\vu)=\<L''(\vu)\bvp,\bvp(\vu)\>+\<L''(\vu),\bvp'(\vu)(\bvp(\vu))\>$. Since $R$ and $\p$ are continuous, $L'(\vff)=0$ and $R(\vff)<0$, there exists $\eps_2\in(0,\eps_1]$ and $\si_2\in(0,\si_1]$ such that \eqref{taylor-exp-1} holds for $\vdu\in B(\vff,\eps_2)$ and $s\in (-\si_2,\si_2)$. On the other hand, it is straightforward to verify that
\[
P(\uu(s)\vdu)\bigg|_{(\vdu,s)=(\vff,0)}=0\q\mbox{and}\q
\frac{\dd}{\dd s}P(\uu(s)\vdu)\bigg|_{(\vdu,s)=(\vff,0)}=\<P'(\vff),\partial_x\vp\>,
\]
where $P$ is defined in Theorem \ref{ground}. We show that $\<P'(\vff),\partial_x\vp\>\neq0$. Otherwise, $\partial_x\vp$ would be tangent to $\n$ at $\vff$, where $\n$ is defined in Theorem \ref{ground}. Hence, $\<L''(\vff)\partial_x\vp,\partial_x\vp\>\geq0$, since $\vff$ minimizes $L$ on $\n$ by Theorem \ref{ground}. But this contradicts Theorem \ref{ground}. Therefore, by the implicit function theorem, there exist $\eps_3\in(0,\eps_2)$
and $\si_3\in(0,\si_2)$ such that for all $\vdu\in B{\vff,\eps_3}$, there exists a unique $s=s(\vdu)\in(-\si_3,\si_3)$ such that $P(\uu(s)\vdu)=0$. Then applying \eqref{taylor-exp-1} to $(\vdu,s(\vdu))\in B{\vff,\eps_3}\times(-\si_3,\si_3)$ and using the fact  $\vff$ minimizes $L$ on $\n$, we have that for $\vdu\in B{\vff,\eps_3}$ there exists $s\in(-\si_3,\si_3)$ such that $S(\vff)\leq L(\uu(s)\vdu)\leq L(\vdu)+\p(\vdu)s$. This inequality can be extended to $\Omega_{\vff,\eps_3}$ from the gauge invariance.

Since $\uu(s)\vdu$ commutes with $\tau_r$, it follows by replacing $\vdu$ with $\uu(s)\vdu$ in \eqref{taylor-exp-1} and then $\delta=-s$ that
\begin{equation}\label{taylor-exp-2}
L(\vff)\leq L(\uu(\delta)\vff)-\p(\uu(\delta)\vff)\delta,
\end{equation}
for all $\delta\in(-\si_3,\si_3)$. Moreover, using Taylor's theorem again and the fact  $\p(\vff)=0$, it follows that
the map $\delta\mapsto L(\uu(\delta)\vff)$ has a strict local maximum at $\delta=0$. Hence, we obtain
\begin{equation}\label{taylor-exp-3}
L(\uu(\delta)\vff)<L(\vff),\q\q\delta\neq0,\;\delta\in(-\si_4,\si_4),
\end{equation}
where $\si_4\in(0,\si_3]$. Thus it follows from \eqref{taylor-exp-2} that
\begin{equation}\label{taylor-exp-4}
\p(\uu(\delta)\vff)<0,\q\q\delta\in(0,\si_4).
\end{equation}

Now let $\delta_j\in(0, \si_4)$ be such that $\delta_j\to0$ as $j\to\infty$, and consider the sequences of initial data $\vec u_{0,j}=\uu(\delta_j)\vff$. It is clear that $\vec u_{0,j}\in H^r(\rr)$, $r > 3/2$ for all positive integers $j$ and $\vec u_{0,j}\to\vff$ in $\x$ as $j\to\infty$. We need only verify that the solution $\vec u_j(t)=\uu(t)\vec u_{0,j}$ of \eqref{E:B6_system} with $\vec u_j(0) = \vec u_{0,j}$ escapes from $\Omega_{\vff,\eps_3}$, for all positive integers $j$ in finite time.  Define
\[
T_j=\sup\left\{t'>0;\,\vec u_j(t)\in \Omega_{\vff,\eps_3},\;\forall t\in(0,t')\right\}
\]
and
\[
\ddd=\left\{\vu\in\Omega_{\vff,\eps_3};\,L(\vu)<L(\vff),\;\p(\vu)<0\right\}.
\]
It follows from \eqref{taylor-exp-1} that for all $j\in\N$ and $t\in (0,T_j)$, there exists $s=s_j(t)\in(-\si_3,\si_3)$ satisfying $L(\vff)\leq L(\vec u_{0,j})+\p(\vec u_j(t))s$. By \eqref{taylor-exp-3} and \eqref{taylor-exp-4}, $u_{0,j}\in\ddd$; and therefore $\vec u_j(t)\in\ddd$ for all $t\in[0,T_j]$. Indeed, if $\p(\vec u_j(t_0))>0$ for some $t_0\in[0, T_j]$, then the continuity of $\p$ implies that there exists some $t_1\in[0, T_j]$ satisfying $\p(\vec u_j(t_1))=0$, and consequently  $L(\vff)\leq L(\vec u_{0,j})$, which contradicts $\vec u_{0,j}\in\ddd$. Hence,  $\ddd$ is bounded away from zero and
\begin{equation}\label{bound}
-\p(\vec u_j)\geq\frac{L(\vff)-L(\vec u_{0,j})}{\si_3}=\eta_j>0,\q\forall t\in [0,T_j].
\end{equation}
Now suppose that for some $j$, $T_j=+\infty$. Then we define a Liapunov function
\begin{equation}\label{Liapunov}
A(t)=\int_\rr\vp(x-\al(\vec u_j))\cdot\vec u_j(x,t)\dd x,\q t\in[0,T_j].
\end{equation}
Since $$\frac{{\rm d} \vec u_j}{{\rm d}t}=\partial_x\kk E'(\vec u_j),$$ then we have
\[
\begin{split}
\frac{{\rm d} A}{{\rm d} t}&=
\left\<\al'(\vec u_j(t)),\frac{{\rm d}\vec u_j}{{\rm d}t}\right\>
\left\<\partial_x\vp(\cdot-\al(\vec u_j(t))),\vec u_j(t)\right\>
+\left\<\vp(\vec u_j(t))),\frac{{\rm d}\vec u_j}{{\rm d} t}\right\>\\
&=\left\<\<\partial_x\vp(\vec u_j(t))),\vec u_j(t)\>\partial_x\kk\al'(\vec u_j(t))
-\partial_x\kk\vp(\vec u_j(t))),E'(\vec u_j(t))\right\>\\
&=-\left\<B_\psi(\vec u_j(t)),L'(u_j(t))\right\>+c\left\<B_\psi(\vec u_j(t)),Q'(\vec u_j(t))\right\>
=-\p(\vec u_j(t)),
\end{split}
\]
for $t\in[0,T_j]$. Therefore it is deduced from \eqref{bound} that
\[
-\frac{\dd A}{\dd t}\geq\eta_j>0,\q\q\forall t\in[0,T_j].
\]
This contradicts the boundedness of $A(t)$ in Lemma \ref{time-estimate}. Consequently $T_j<+\infty$ for all $j$, which means that $\vec u_j$ eventually leaves $\Omega_{\vff,\eps_3}$. This completes the proof.
\fim

The remaining results of this section are applications of Theorem \ref{T:instability}. In verifying the hypotheses of this theorem, we will use the fact that for any $\vec w_1=(u_1,v_1)$ and $\vec w_2=(u_2,v_2)$ in $\x$ we have
    \begin{equation}\label{E:H_formula}
    \left\langle H\vec w_1,\vec w_2\right\rangle
    =\int_\rr(u_1''''+\beta u_1''+(1-c^2)u_1-f'(\ff)u_1)u_2+(cu_1+v_1)(cu_2+v_2)\dd x.
    \end{equation}
In view of this, we define $H_1=\partial_x^4+\beta\partial_x^2+(1-c^2)+f'(\ff)$. Our first result is the following complement of Theorem \ref{T:stability}.

\begin{theorem}\label{T:instability_d}
Suppose $\beta<2$ and assume there exists a $C^2$ map $c\mapsto\varphi\in\g$ for $c<c_*$. If $d''(c)<0$, then $\mathcal{O}_{\vff}$ is $\x$-unstable for any $\vec\ff\in\g$.
\end{theorem}

\proof Define
\[
\vec\psi(x)=\int_{-\infty}^x\vec\ff(y)-\frac{2d'(c)}{d''(c)}\vec\ff_c(y)\dd y,
\]
where $\vec\ff_c=\frac{d}{dc}\vec\ff=(\ff_c,-c\ff_c-\ff)$.
We need to show that $\vec\psi$ satisfies the hypotheses of Theorem \ref{T:instability}. Now
    \[
    \left\langle Q'(\vec\ff),\partial_x\vec\psi\right\rangle=
    \left\langle Q'(\vec\ff),\vec\ff\right\rangle-\frac{2d'(c)}{d''(c)}
    \frac{d}{dc}Q(\vec\ff)
    =2d'(c)=2d'(c)=0
    \]
so the first hypothesis is satisfied. To show that the second hypothesis is satisfied, first note that
\begin{align*}
  \left\langle H\partial_x\vec\psi,\partial_x\vec\psi\right\rangle
  &=\left\langle H\vec\ff,\vec\ff\right\rangle-\frac{4d'(c)}{d''(c)}\left\langle H\vec\ff,(\vec\ff)_c\right\rangle
  +\left(\frac{2d'(c)}{d''(c)}\right)^2\left\langle H\vec\ff_c,\vec\ff_c\right\rangle.
\end{align*}
Using the homogeneity of $f$ and the solitary wave equation, we have
    \[
    H_1(\ff)=f(\ff)-f'(\ff)\ff=(1-p)f(\ff),
    \]
so by relation \eqref{E:H_formula} it follows that
    \[
    \left\langle H\vec\ff,\vec\ff\right\rangle
    =(1-p)\int_\rr f(\ff)\ff\dd x=(1-p)(p+1)\int_\rr F(\ff)\dd x
    =-2(p+1)d(c)
    \]
and
    \[
    \left\langle H\vec\ff,\vec\ff_c\right\rangle
    =(1-p)\int_\rr f(\ff)\ff_c\dd x=(1-p)\left(\int_\rr F(\ff)\dd x\right)_c
    =-2d'(c).
    \]
By differentiating the solitary wave equation with respect to $c$, it follows that
    \[
    H_1\ff_c=2c\ff,
    \]
so
    \begin{align*}
    \left\langle H\vec\ff_c,\vec\ff_c\right\rangle
    &=\int_\rr 2c\ff\ff_c+\ff^2\dd x\\
    &=c\left(\int_\rr\ff^2\dd x\right)_c-\frac{d'(c)}{c}\\
    &=-c\left(\frac{d'(c)}{c}\right)_c-\frac{d'(c)}{c}\\
    &=-d''(c).
    \end{align*}
It now follows that
    \begin{align*}
    \left\langle H\partial_x\vec\psi,\partial_x\vec\psi\right\rangle
    &=-2(p+1)d(c)-\frac{4d'(c)}{d''(c)}(-2d'(c))
    +\left(\frac{2d'(c)}{d''(c)}\right)^2(-d''(c))\\
    &=-2(p+1)d(c)+\frac{4(d'(c))^2}{d''(c)}<0
    \end{align*}
since $d''(c)<0$. This completes the proof.
\fim

We next apply Theorem \ref{T:instability} to obtain conditions on $p$, $\beta$ and $c$ that imply orbital instability. For our choices of unstable direction we will use the following.
\begin{enumerate}[(i)]
  \item $\vec\psi_x=(\ff,+c\ff)$ -- for small $c$, and any $p>1$.
  \item $\vec\psi_x=\vec\ff+2x\vec\ff_x$ -- for large $p$.
\end{enumerate}

\begin{lemma}\label{L:unstable_direction} Let $\partial_x\vec\psi=\vec\ff+2x\vec\ff_x
=(\ff+2x\ff_x,-c(\ff+2x\ff'))$. Then $\left\langle Q'(\vec\ff),\partial_x\vec\psi\right\rangle=0$ and
    \[
    \left\langle H\partial_x\vec\psi,\partial_x\vec\psi\right\rangle=\frac{(1-p)(p-3)}{p+1}K(\ff)+\int_\rr 24(\ff'')^2-4\beta(\ff')^2\dd x.
    \]
\end{lemma}

\proof
First, we have
    \[
      \left\langle Q'(\vec\ff),\partial_x\vec\psi\right\rangle=\int_\rr -2c\ff(\ff+2c\ff')\dd x=0
    \]
as claimed. Next we have
    \[
    \left\langle \partial_xH\vec\psi,\partial_x\vec\psi\right\rangle=\int_\rr
    [(\ff+2x\ff')''''+\beta (\ff+2x\ff')''+(1-c^2)(\ff+2x\ff')-f'(\ff)(\ff+2x\ff')](\ff+2x\ff')\dd x,
    \]
which may be split into three terms:
\begin{align*}
  A_1&=\int_\rr
    [\ff''''+\beta \ff''+(1-c^2)\ff-f'(\ff)\ff]\ff\dd x,\\
    A_2&=2\int_\rr
    [\ff''''+\beta \ff''+(1-c^2)\ff-f'(\ff)\ff](2x\ff')\dd x,\\
    A_3&=\int_\rr
    [(2x\ff')''''+\beta (2x\ff')''+(1-c^2)(2x\ff')-f'(\ff)(2x\ff')](2x\ff')\dd x.
\end{align*}
Since $\ff''''+\beta \ff''+(1-c^2)\ff-f'(\ff)\ff=(1-p)f(\ff)$ we have
    \[
    A_1=(1-p)(p+1)\int_\rr F(\ff)\dd x=(1-p)K(\ff)
    \]
and
    \[
    A_2=-4(1-p)\int_\rr F(\ff)\dd x=-\frac{4(1-p)}{p+1}K(\ff).
    \]
For $A_3$ first observe that by differentiating \eqref{E:B6_solitary} we obtain $\ff'''''+\beta \ff'''+(1-c^2)\ff'-f'(\ff)\ff'$, and thus
    \[
    (2x\ff')''''+\beta (2x\ff')''+(1-c^2)(2x\ff')-f'(\ff)(2x\ff')
    =8\ff''''+4\beta\ff''.
    \]
Thus
    \[
    A_3=\int_\rr (8\ff''''+4\beta\ff'')2x\ff'\dd x=\int_\rr 24(\ff'')^2-4\beta(\ff')^2\dd x
    \]
so summing $A_1$, $A_2$ and $A_3$ yields the result of the lemma.
\fim

\begin{theorem}\label{T:instability_criteria} Suppose $c^2<1$, $\beta<\beta_*=2\sqrt{1-c^2}$ and $\ff\in\g$. Recall that $c_\ast=\sqrt{1-\beta_+^2/4}$. 
Then $\mathcal{O}(\ff)$ is $\x$-unstable in the following cases.
\begin{enumerate}[(i)]
\item $p>1$ and $c^2<\left(\frac{p-1}{p+3}\right)c_\ast^2$.
  \item $p\geq9$, $c^2<1$ and $\beta<\left(\frac{(p-1)(p-9)}{(p-1)^2+16}\right)\beta_*$.
\end{enumerate}
\end{theorem}

\begin{figure}
  \begin{center}
    \scalebox{0.45}{
    \includegraphics{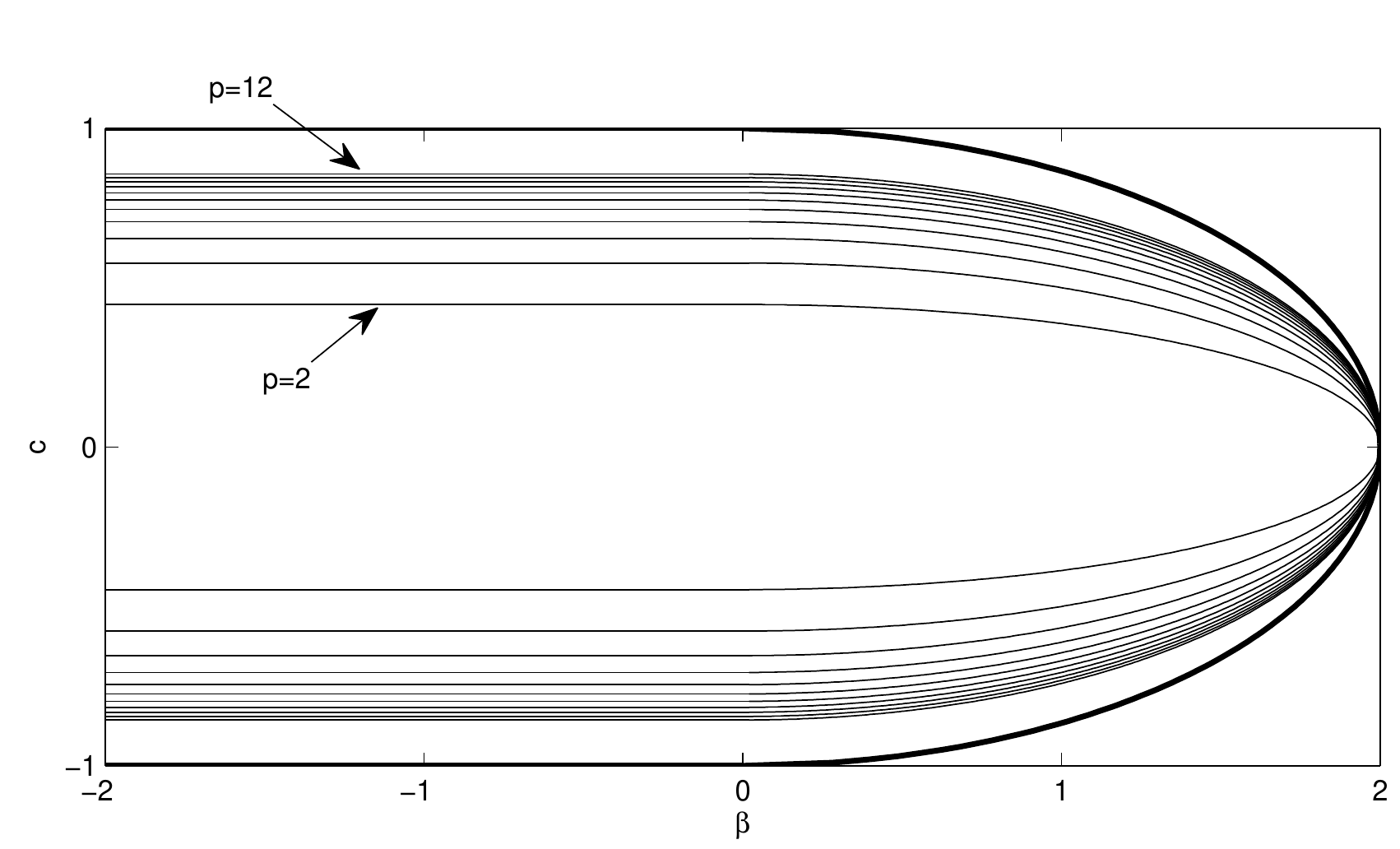}\quad
    \includegraphics{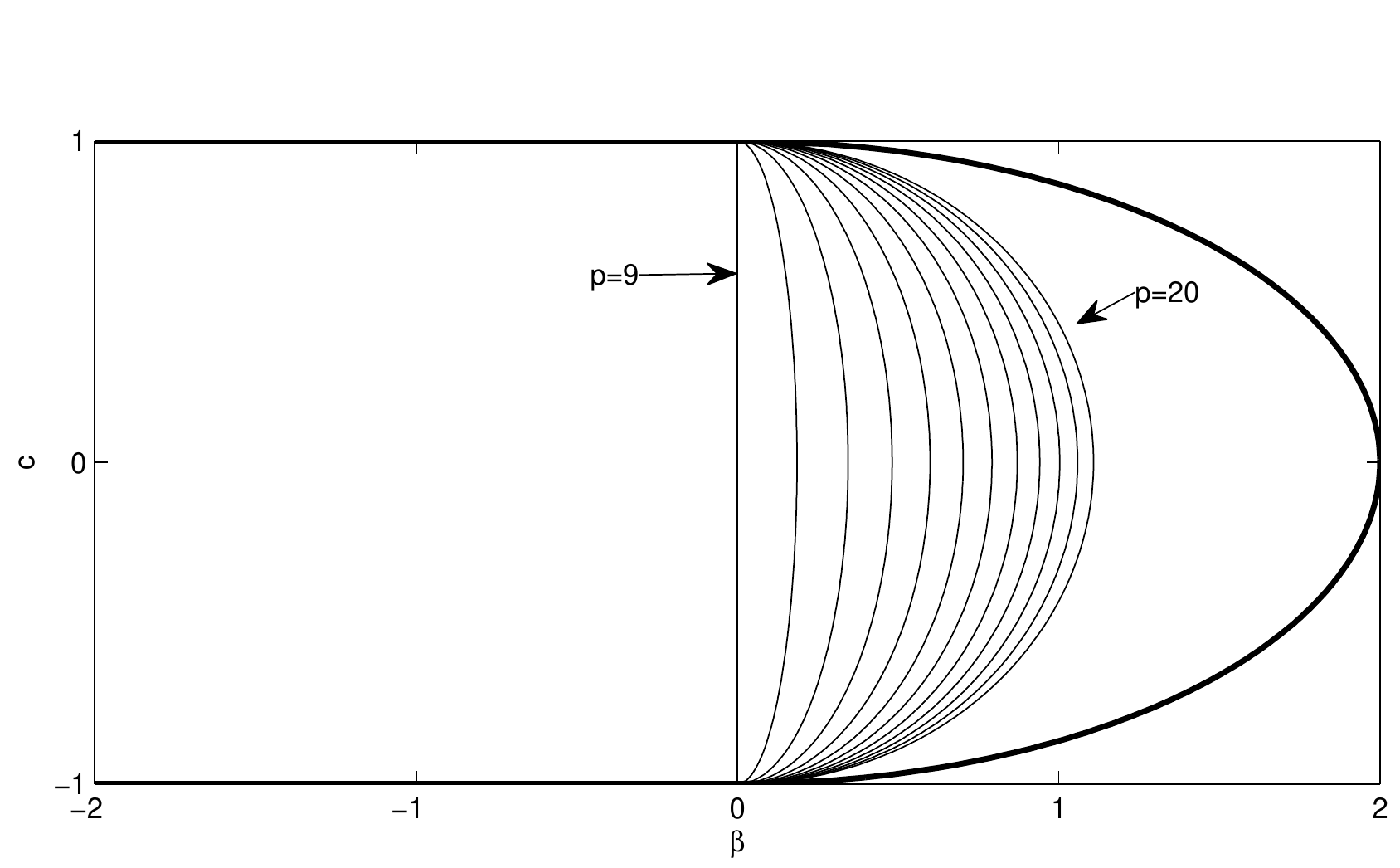}
    }
  \end{center}
  \caption{The regions of instability guaranteed by Theorem \ref{T:instability_criteria}. The regions described in part (i) lie between the upper and lower curves on the first plot. The regions described in part (ii) lie to the left of the curves in the second plot. Both regions grow to fill the domain of $d$ as $p$ increases.}
\end{figure}

\proof To prove the first statement, consider the choice $\partial_x\vec\psi=(\ff,c\ff)$. It is easy to see that $\left\langle Q'(\vec\ff),\partial_x\vec\psi\right\rangle=0$. Next we compute
    \begin{align*}
    \left\langle \partial_xH\vec\psi,\partial_x\vec\psi\right\rangle&=
    \int_\rr(\ff''''+\beta\ff''+\ff-f'(\ff)\ff)\ff+(c\ff)^2+2c\ff c\ff\dd x\\
    &=\int_\rr(\ff''''+\beta\ff''+(1-c^2)\ff-f'(\ff)\ff)\ff+4c^2\ff^2\dd x\\
    &=(1-p)K(\ff)+4c^2\int_\rr\ff^2\dd x.
    \end{align*}
First suppose $\beta\leq 0$, in which case $c_*=1$. Then $I(\ff)\geq(1-c^2)\int_\rr \ff^2\dd x$, so
    \[
    \left\langle \partial_xH\vec\psi,\partial_x\vec\psi\right\rangle\leq \left(1-p+\frac{4c^2}{1-c^2}\right)I(\ff).
    \]
Now suppose $\beta>0$. Then
    \begin{align*}
    I(\ff)&=\int_\rr(\ff'')^2-\beta(\ff')^2+\frac{\beta^2}{4}\ff^2\dd x
    +\left(1-c^2-\frac{\beta^2}{4}\right)\int_\rr\ff^2\dd x\\
    &\geq\left(1-c^2-\frac{\beta^2}{4}\right)\int_\rr\ff^2\dd x
    \end{align*}
and thus
    \[
    \left\langle \partial_xH\vec\psi,\partial_x\vec\psi\right\rangle\leq \left(1-p+\frac{4c^2}{1-c^2-\frac{\beta^2}{4}}\right)I(\ff).
    \]
Hence for any $\beta<2$ we have 
    \[
    \left\langle \partial_xH\vec\psi,\partial_x\vec\psi\right\rangle\leq \left(1-p+\frac{4c^2}{c_\ast^2-c^2}\right)I(\ff),
    \]
and this quantity is negative when condition (i) is satisfied.

To prove (ii), we use the choice of unstable direction given in Lemma \ref{L:unstable_direction}.
Multiplying the solitary wave equation by $x\ff'$ and integrating yields the Pohozaev identity
    \begin{equation}
      \int_\rr 3(\ff'')^2-\beta(\ff')^2-(1-c^2)\ff^2+2F(\ff)\dd x=0.
    \end{equation}
The identity $I(\ff)=K(\ff)$ may be written
    \begin{equation}
      \int_\rr (\ff'')^2-\beta(\ff')^2+(1-c^2)\ff^2-(p+1)F(\ff)\dd x=0.
    \end{equation}
Together these give
    \[
    \int_\rr 4(\ff'')^2-2\beta(\ff')^2-(p-1)F(\ff)\dd x=0.
    \]
Together with the result of Lemma \ref{L:unstable_direction}, this gives
    \[
    \left\langle \partial_xH\vec\psi,\partial_x\vec\psi\right\rangle=\frac{(1-p)(p-9)}{p+1}K(\ff)+8\beta\int_\rr(\ff')^2\dd x.
    \]
Since
    \[
    I(\ff)\geq (\beta_*-\beta)\int_\rr(\ff')^2\dd x
    \]
it then follows that
    \[
    \left\langle \partial_xH\vec\psi,\partial_x\vec\psi\right\rangle\leq \left(\frac{8\beta}{\beta_*-\beta}+\frac{(1-p)(p-9)}{p+1}\right)K(\ff).
    \]
The term in parentheses is negative when $\beta$ satisfies condition (ii) above.
\fim

\section{Further Properties of $d$.}\label{S:d_properties}

In this section we establish further properties of the function $d$.
We first obtain bounds on the function $d$ as $c$ approaches $\pm c_*=\sqrt{1-\beta_+^2/4}$. To obtain these bounds, we use trial functions to obtain bounds on the Rayleigh quotient that defines $m(\beta,c)$. To motivation the choice of trial function, we observe that solutions of the solitary wave equation \eqref{E:B6_solitary} have tails that decay like solutions of the linear equation
\begin{equation}\label{E:linear_solitary}
  \ff''''+\beta\ff''+(1-c^2)\ff=0.
\end{equation}
The fundamental solution of this equation is the function $h$ defined by \eqref{E:kernel_definition}. Recalling that $h$ is given explicitly by the expressions in \eqref{E:kernel_formulas}, we see that $h\in H^2(\rr)$, and is thus a valid trial function provided $K(h)>0$. The fact that $K(h)>0$ will be verified below. Since scaling has no effect on the Rayleigh quotient that defines $m$, we use the following scaled versions of $h$ for simplicity. If $-\beta_\ast<\beta<\beta_\ast$, define
    \begin{equation}\label{E:trial_u}
    u(x)=e^{-\sigma |x|}(\omega\cos(\omega x)+\sigma\sin(\omega |x|))
    \end{equation}
and if $\beta<-\beta_*$ define
    \begin{equation}\label{E:trial_v}
    v(x)=\lambda_2e^{-\lambda_1|x|}-\lambda_1e^{-\lambda_2|x|}
    \end{equation}
where $\lambda_1$, $\lambda_2$, $\sigma$ and $\omega$ are defined by \eqref{E:lambda_sigma_omega}.

\begin{theorem}\label{T:d_bound} Suppose $f(u)=|u|^{p-1}u$. Fix $\beta<2$. Then
    \[
    d(\beta,c)=O\left((c_*-c)^\frac{p+3}{2(p-1)}\right)
    \]
  as $c$ approaches $c_*$.
\end{theorem}

\proof First consider $0\leq\beta<2$. Then $c_*=\sqrt{1-\frac14\beta^2}$, and it follows that
\begin{align*}
  \sigma&=O(\sqrt{c_*-c})\\
  \omega&=\sqrt{\beta/2}+O(c_*-c)
\end{align*}
as $c\to c_*$. For the trial function $u$ given by \eqref{E:trial_u}, a direct calculation reveals that
$I(u)=4\sigma\omega^2(\sigma^2+\omega^2)$, and by calculations similar to those in \cite{levandosky1} we have
    \[
    K(u)=\int_\rr|u|^{p+1}\dd x\geq\frac1{O(\sigma)}
    \]
for small $\sigma>0$. Thus
    \[
    m(\beta,c)\leq\frac{I(u)}{K(u)^{2/(p+1)}}=O\left(\sigma^\frac{p+3}{p+1}\right)
    =O\left((c_*-c)^\frac{p+3}{2(p+1)}\right)
    \]
as $c\to c_*$.

Next, when $\beta<0$ we have $c_*=1$, and
\begin{align*}
  \lambda_1&=O(\sqrt{c_*-c})\\
  \lambda_2&=\sqrt{-\beta}+O(c_*-c)
\end{align*}
as $c\to c_*$. For the trial function $v$ given by \eqref{E:trial_v}, another direct calculation reveals that $I(v)=2(\lambda_2-\lambda_1)\lambda_1\lambda_2(\lambda_2^2-\lambda_1^2)$, and by calculations similar to those in \cite{levandosky1} we have
    \[
    K(v)\geq\frac1{O(\lambda_1)}
    \]
for small $\lambda_1>0$, and thus
    \[
    m(\beta,c)\leq\frac{I(v)}{K(v)^{2/(p+1)}}=O\left(\lambda_1^\frac{p+3}{p+1}\right)
    =O\left((c_*-c)^\frac{p+3}{2(p+1)}\right)
    \]
as $c\to c_*$.

The result then follows by the relation between $d$ and $m$.
\fim

\begin{corollary}\label{C:stability} Suppose $f(u)=|u|^{p-1}u$ where $1<p<5$. Fix $\beta<2$. Then there exist $c$ arbitrarily close to $c_*$ such that $\g$ is $\x$-stable.
\end{corollary}

\proof  Since $\frac{p+3}{2(p-1)}>1$ when $1<p<5$, the function $(c_*-c)^\frac{p+3}{2(p-1)}$ is convex and vanishes at $c=c_*$. Thus by Theorem \ref{T:d_bound}, $d$ vanishes at $c=c_*$ and is bounded above by a convex function. Since $d$ is positive, this implies that there exist $c$ arbitrarily close to $c_*$ such that
$d''(c)>0$, and the result then follows from Theorem \ref{T:stability}.
\fim

\begin{remark} The results of Theorem \ref{T:d_bound} and Corollary \ref{C:stability} also hold for the even nonlinearity $f(u)=|u|^{p+1}$ in the case that $\beta<0$ since the trial function $v$ is positive for small $\lambda_1$ ($c$ near $1$). However, for $0\leq\beta<2$ the non-positivity of $u$ only allows one to obtain the weaker estimate $d(\beta,c)=O(\sqrt{c_*-c})$ which does not imply convexity of $d$ near $c_*$.
\end{remark}

We next present the main scaling identity satisfied by the function $d$.

\begin{theorem}\label{T:d_scaling} Let $c^2<1$ and $\beta<\beta_*=2\sqrt{1-c^2}$. Then for any $0<r\leq(1-c^2)^{-1/2}$ we have
    \[
    d(r\beta,\sqrt{1-r^2(1-c^2)})=r^\frac{3p+5}{2(p-1)}d(\beta,c).
    \]
\end{theorem}

\proof Recall that
    \[
    m(\beta,c)=\inf\left\{
    \frac{I(u)}{K(u)^{2/(p+1)}}
    \right\}
    \]
where
    \[
    I(u)\equiv I(u;\beta,c)=\int_\rr u_{xx}^2-\beta u_x^2+(1-c^2)u^2\dd x
    \]
and
    \[
    K(u)=(p+1)\int_\rr F(u)\dd x.
    \]
Given any $u\in H^2(\rr)$, we set $v(x)=r^{3/4}u(r^{-1/2}x)$. Then
    \[
    I(v;\beta,c)=\int_\rr u_{xx}^2-r\beta u_x^2+r^2(1-c^2)\dd x=I(u;r\beta,\sqrt{1-r^2(1-c^2)})
    \]
and $K(v)=r^\frac{3p+5}{4}K(u)$. If we then suppose $v$ achieves the minimum $m(\beta,c)$ it follows that
    \[
    m(\beta,c)=\frac{I(v;\beta,c)}{K(v)^{2/(p+1)}}=r^{-\frac{3p+5}{2(p+1)}}\frac{I(u;r\beta,\sqrt{1-r^2(1-c^2)})}{K(u)^{2/(p+1)}}
    \geq r^{-\frac{3p+5}{2(p+1)}}m(r\beta,\sqrt{1-r^2(1-c^2)}).
    \]
By supposing that $u$ achieves the minimum $m(r\beta,\sqrt{1-r^2(1-c^2)})$ we obtain the reverse inequality, and the result then follows by the relation between $d$ and $m$.
\fim

\begin{remark}
This scaling property implies that all values of $d$ on any semi-ellipse $\beta=k\sqrt{1-c^2}$ with $k<2$ are determined by any single value of $d$ on that semi-ellipse.
\end{remark}

Setting $\beta=0$ and $r^2=(1-c^2)^{-1}$ in Theorem \ref{T:d_scaling} gives the following result.

\begin{corollary} When $\beta=0$, $d(c)=(1-c^2)^\frac{3p+5}{4(p-1)}d(0)$, and it follows that
\begin{enumerate}[(i)]
  \item If $p\geq9$, then $d''(c)<0$ for $c^2<1$,
  \item If $p<9$, then $d''(c)<0$ for $c^2<\frac{2(p-1)}{p+7}$ and $d''(c)>0$ for $c^2>{2(p-1)}{p+7}$.
\end{enumerate}
\end{corollary}

\begin{figure}
  \begin{center}
    \scalebox{0.6}{\includegraphics{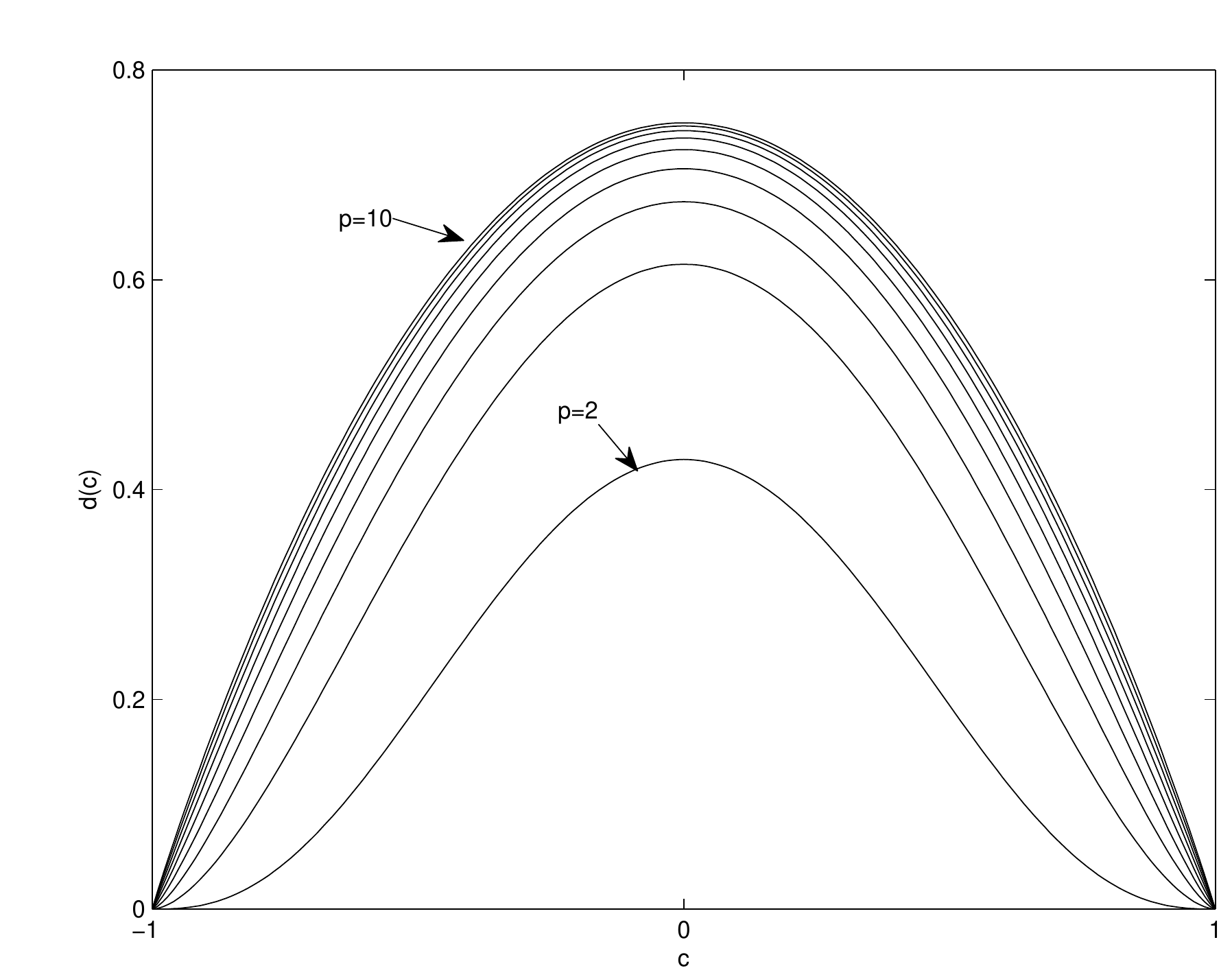}}
  \end{center}
  \caption{Plots of $d(c)$ for $\beta=0$ and $p=2,3,\ldots,10$. The values of $d(0)$ were found via the numerical methods described in Section \ref{S:numerical}.}
\end{figure}

\begin{theorem}\label{T:d_sign_change} Suppose $d$ is twice differentiable on its domain, and consider the curve
$\Gamma_k=\{(\beta,c):0\leq c<1,\beta=k\sqrt{1-c^2}\}$ for some $k<2$. Then $d_{cc}(\beta,c)$ changes sign at most once along $\Gamma_k$.
\end{theorem}

\proof We present two proofs of this fact, both of which make use of the scaling property of $d$.
First, setting $r=(1-c^2)^{-1/2}$ in Theorem \ref{T:d_scaling} gives
    \[
    d(\beta,c)=(1-c^2)^\gamma d(\beta/\sqrt{1-c^2},0)
    \]
where $\gamma=\frac{3p+5}{4(p-1)}$. Equivalently, setting $s=\sqrt{1-c^2}$ we have
    \[
    d(\beta,\sqrt{1-s^2})=s^{2\gamma} d(\beta/s,0).
    \]
Differentiating once with respect to $s$ gives
    \[
    d_c(\beta,\sqrt{1-s^2})\cdot\frac{-s}{\sqrt{1-s^2}}=2\gamma s^{2\gamma-1}d(\beta/s,0)-\beta s^{2\gamma-2} d_\beta(\beta/s,0).
    \]
or equivalently
    \[
    d_c(\beta,\sqrt{1-s^2})=\sqrt{1-s^2}\left[-2\gamma s^{2\gamma-2}d(\beta/s,0)+\beta s^{2\gamma-3} d_\beta(\beta/s,0)\right].
    \]
Differentiating again with respect to $s$ then gives
    \[\begin{split}
    d_{cc}(\beta,\sqrt{1-s^2})\cdot\frac{-s}{\sqrt{1-s^2}}&=\frac{-s}{\sqrt{1-s^2}}\left[-2\gamma s^{2\gamma-2}d(\beta/s,0)+\beta s^{2\gamma-3} d_\beta(\beta/s,0)\right]\\
    &\qquad+\sqrt{1-s^2}\left[-2\gamma(2\gamma-2) s^{4\gamma-3}d(\beta/s,0)\right.\\
&    \left.\qquad+\beta(2\gamma-3)s^{\gamma-4} d_\beta(\beta/s,0)-\beta^2s^{2\gamma-5}d_{\beta\beta}(\beta/s,0)\right].
    \end{split}\]
Now denote $\beta_0=\beta/\sqrt{1-c^2}=\beta/s$. Then this becomes
    \[\begin{split}
    d_{cc}(\beta,c)&=-2\gamma s^{2\gamma-2}d(\beta_0,0)+\beta s^{2\gamma-3} d_\beta(\beta_0,0)\\
    &\quad+(1-s^2)\left[2\gamma(2\gamma-2) s^{2\gamma-4}d(\beta_0,0)-\beta(4\gamma-3)s^{2\gamma-5} d_\beta(\beta_0,0)+\beta^2s^{2\gamma-6}d_{\beta\beta}(\beta_0,0)\right]\\
    &=-2\gamma s^{2\gamma-2}d(\beta_0,0)+\beta_0 s^{2\gamma-2} d_\beta(\beta_0,0)\\
    &\quad+c^2\left[2\gamma(2\gamma-2) s^{2\gamma-4}d(\beta_0,0)-\beta_0(4\gamma-3)s^{2\gamma-4} d_\beta(\beta_0,0)+\beta_0^2s^{2\gamma-4}d_{\beta\beta}(\beta_0,0)\right]
    \end{split}\]
Simplification yields
     \[\begin{split}
    d_{cc}(\beta,c)&=s^{2\gamma-4}\left[2\gamma d(\beta_0,0)(c^2(2\gamma-1)-1)+\beta_0d_\beta(\beta_0,0)(1-(4\gamma-2)c^2)+c^2\beta_0^2d_{\beta\beta}(\beta_0,0)\right]\\
    &=s^{2\gamma-4}\left[c^2(2\gamma(2\gamma-1)d(\beta_0,0)-2(2\gamma-1)\beta_0d_\beta(\beta_0,0)+\beta_0^2d_{\beta\beta}(\beta_0,0))\right.\\
    &\left.\qquad
    +(\beta_0d_\beta(\beta_0,0)-2\gamma d(\beta_0,0))\right]
    \end{split}\]
Since the bracketed term is linear in $c^2$, this shows that $d_{cc}$ changes sign at most once on $\Gamma_k$, and the change of sign occurs when $c=\sqrt{P}$, where
    \[
    P=P(\beta_0,\gamma)=\frac{-\beta_0d_\beta(\beta_0,0)+2\gamma d(\beta_0,0)}{2\gamma(2\gamma-1)d(\beta_0,0)-2(2\gamma-1)\beta_0d_\beta(\beta_0,0)+\beta_0^2d_{\beta\beta}(\beta_0,0)}
    \]
provided $0<P<1$.

Alternately choose any point $(\beta_0,c_0)\in\Gamma_k$ with $c_0\neq0$. Then applying Theorem \ref{T:d_scaling} with $r=\beta_0/\beta$ gives
    \[
    d(\beta,c)=\left(\frac{\beta}{\beta_0}\right)^q d(\beta_0,c_0)
    \]
where $q=\frac{3p+5}{2(p-1)}$. Differentiating twice with respect to $c$ and using the relation
$c_0=\sqrt{1-\beta_0^2(1-c^2)/\beta^2}$ we have
    \[
    d_{cc}(\beta,c)=\frac1{c_0^3(1-c_0^2)}\left(\frac{\beta}{\beta_0}\right)^{q-4}
    \left[(1-c_0^2)c_0c^2d_{cc}(\beta_0,c_0)+(c_0^2-c^2)d_c(\beta_0,c_0)\right].
    \]
The term outside the brackets is positive, while the bracketed term is linear in $c^2$ and therefore can change sign at most once for $0<c<1$. The change of sign occurs when $c=\sqrt{P}$, where
    \[
    P=P(c_0)=\frac{c_0^2d_c(\beta_0,c_0)}{(c_0^2-1)c_0d_{cc}(\beta_0,c_0)+d_c(\beta_0,c_0)}
    \]
provided $0<P<1$.
\fim

\begin{remark} Theorem \ref{T:d_sign_change} does {\em not} imply that for $\beta$ fixed $d_{cc}$ has at most one sign change as $c$ varies. Indeed, when $p=4$ there exist $\beta$ for which $d_{cc}$ changes sign three times as $c$ varies from $0$ to $c_\ast$. See Figure \ref{F:d_cc_sign_change}.
\end{remark}

\begin{figure}
  \begin{center}
    \scalebox{0.5}{\includegraphics{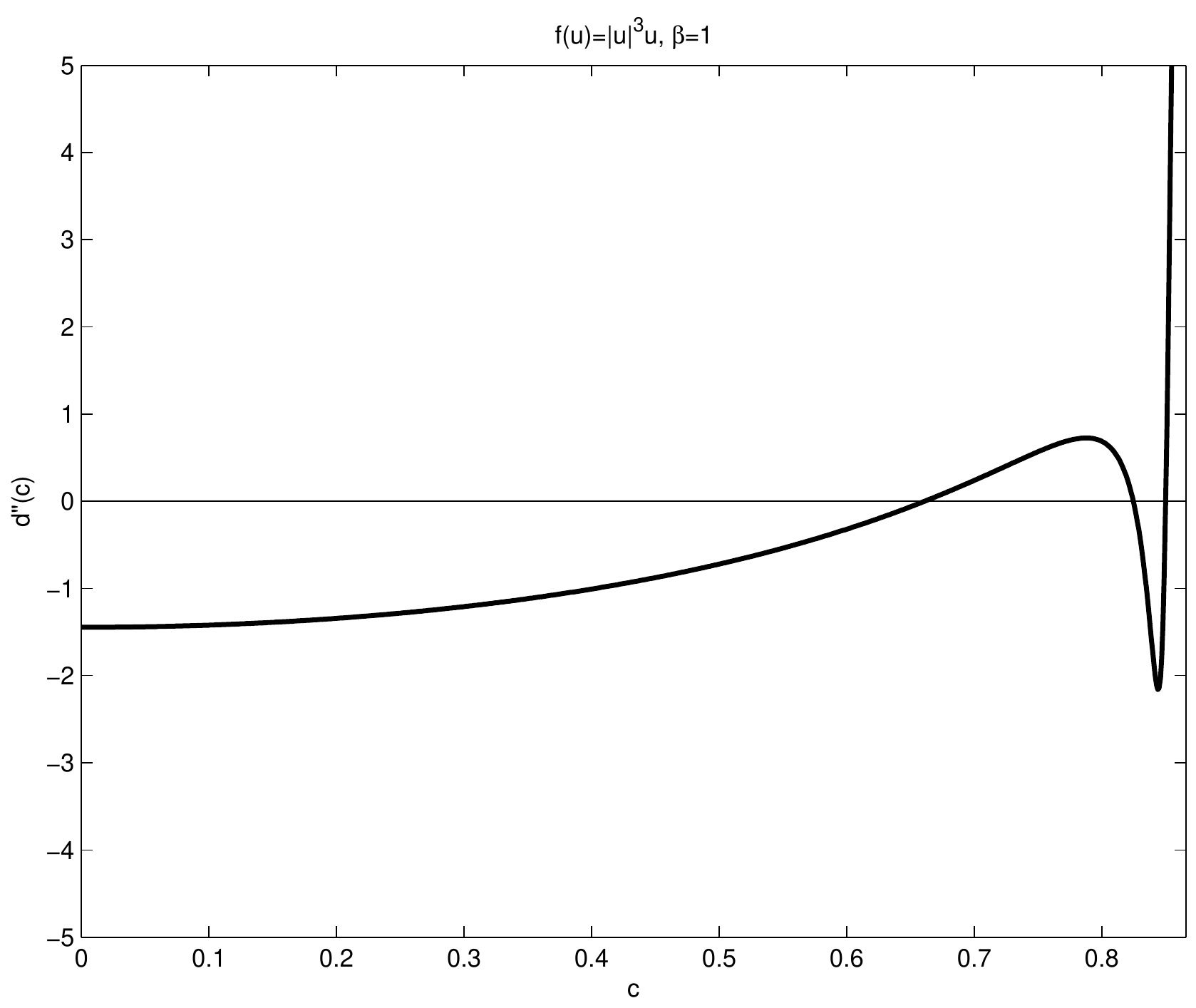}}
  \end{center}
  \caption{When $f(u)=|u|^3u$ and $\beta=1$, the sign of $d''(c)$ changes sign three times.}\label{F:d_cc_sign_change}
\end{figure}

\section{Numerical Results}\label{S:numerical}

In this section we present numerical calculations of $d$ and its derivatives for the nonlinearities $f(u)=|u|^{p}$ and $f(u)=|u|^{p-1}u$ for several values of $p$. These results illustrate precisely the regions in the $(\beta,c)$-plane where $d_{cc}$ is positive and negative, hence where the solitary waves are stable or unstable.

The method consists of numerically computing a solitary wave $\ff$ for given $(\beta,c)$ and using the relations
    \begin{align*}
      d(\beta,c)&=\frac{p-1}{2(p+1)}K(\ff)\\
      d_\beta(\beta,c)&=-\frac12\int_\rr\ff_x^2\dd x\\
      d_c(\beta,c)&=-c\int_\rr\ff^2\dd x
    \end{align*}
to compute $d$ and its first derivatives. By then doing this for several values of $(\beta,c)$ the second derivatives $d_{cc}$ and $d_{\beta\beta}$ may be calculated numerically. By the scaling relation in Theorem \ref{T:d_scaling}, it suffices to perform these calculations over the segments
\begin{align*}
  S_1&=\{(\beta,c):c=0, -1\leq\beta<2\}\\
  S_2&=\{(\beta,c):\beta=-1,0\leq c<1\},
\end{align*}
since for every $k<2$ the semi-ellipse $\Gamma_k=\{(\beta,c):0\leq c<1,\beta=k\sqrt{1-c^2}\}$ passes through either $S_1$ or $S_2$. The calculations in the proof of Theorem \ref{T:d_sign_change} may then be used to determine the locations where $d_{cc}$ changes sign.

\begin{figure}
  \begin{center}
    \scalebox{0.7}{\includegraphics{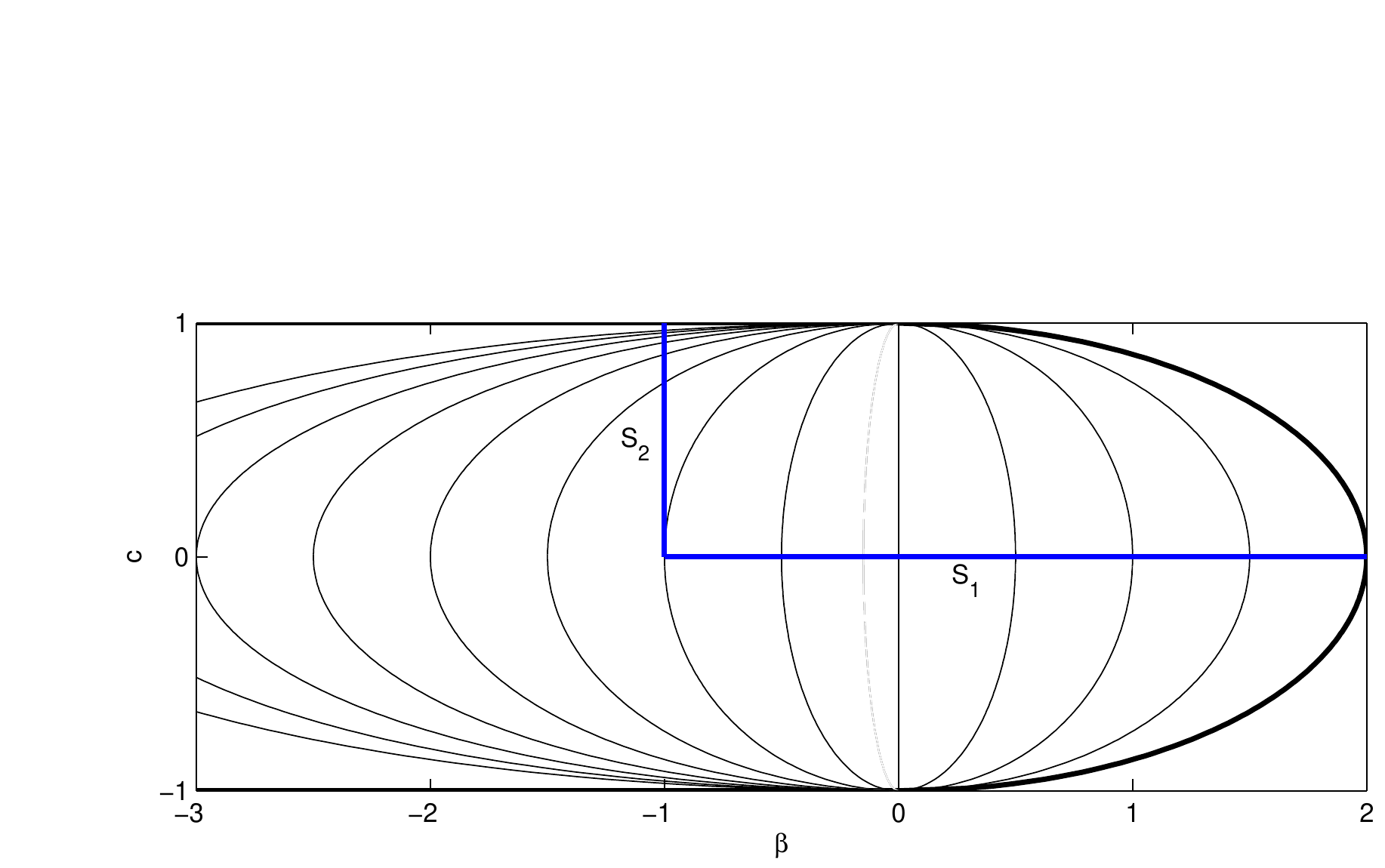}}
  \end{center}
  \caption{The domain of $d$, $\{(\beta,c):c^2<1,\beta<2\sqrt{1-c^2}\}$. Also shown are the semi-ellipses  $\Gamma_k$ along which the scaling relation determines the values of $d$, and the segments $S_1$ and $S_2$ along which the numerical calculations were performed.}\label{F:d_domain}
\end{figure}

To compute the solitary waves, the following spectral method due to Petviashvili. The Fourier transform of the solitary wave equation \eqref{E:B6_solitary} is
    \[
    (\xi^4-\beta\xi^2+(1-c^2))\hat{\ff}=\widehat{f(\ff)}
    \]
so we perform the iteration
    \[
    \hat{\ff}_{k+1}=M^{p/(p-1)}\frac{\widehat{f(\ff_k)}}{\xi^4-\beta\xi^2+(1-c^2)}
    \]
where the stabilizing factor $M$ is given by
    \[
    M=\frac{\int_\rr (\xi^4-\beta\xi^2+(1-c^2))|\hat{\ff}_k|^2\,d\xi}{\int_\rr\widehat{f(\ff_k)}\overline{\hat{\ff}_k}\,d\xi}.
    \]
The convergence properties of this method were studied in \cite{pelinovsky_stepanyants}, where it was shown that the exponent $\frac{p}{p-1}$ of the stabilizing factor $M$ yields the fastest rate of convergence. In the case of the nonlinearity $f(u)=u^p$ for integer $p$, there exist exact solutions of the form
    \[
    \ff(x)=\left(\frac{(p+3)(3p+1)}{8(p+1)}\right)^{\frac1{p-1}}
    \operatorname{sech}^{\frac4{p-1}}\left(\frac{p-1}{4(p+1)}\sqrt{-(p^2+2p+5)/\beta}x\right)
    \]
when $\beta=-\left(\frac{p+1}{2}+\frac2{p+1}\right)\sqrt{1-c^2}$ (\cite{dey_khare_kumar}). On the spatial domain $[-200,200]$ the numerically computed solitary waves very closely approximate the exact solutions,  with an $L^2$ error on the order of $10^{-6}$ after about 100 iterations using Gaussian initial data.

The results of these computations for the odd nonlinearity $f(u)=|u|^{p-1}u$ and even nonlinearity $f(u)=|u|^p$ are shown in Figures \ref{F:dcc_nodal_odd} and \ref{F:dcc_nodal_even}, respectively. Each curve corresponds to a different choice of the power $p$, and separates the domain $D^+$ into two regions
\begin{align*}
  D_u&=\{(\beta,c)\in D^+:d_{cc}(\beta,c)<0\}\\
  D_s&=\{(\beta,c)\in D^+:d_{cc}(\beta,c)>0\}.
\end{align*}
Since $d_{cc}(\beta,0)<0$ for all $\beta$, the region of unstable solitary waves, $D_u$, is the ``lower'' region that contains the $\beta$-axis, while the region of stable solitary waves, $D_s$, is the remaining region. Several observations may be made regarding the stable and unstable regions.
\begin{obs}\label{O:numerical_results}
\begin{enumerate}[(i)]
\item For $p<5$, the stable region $D_s$ is unbounded and for each fixed $\beta$ contains points $(\beta,c)$  near $(\beta,c_*)$, in agreement with the result of Corollary \ref{C:stability}.
\item For $p\geq5$, the stable region $D_s$ is bounded, and when $p>5$ appears to consist of the set of points interior to a smooth closed curve that passes through $(0,1)$.
\item For $p\geq12$, $D_s$ is empty.
\item For sufficiently large $p$, there exist $\beta$ such that $d_{cc}$ changes sign more than once as $c$ varies from $0$ to $c_*$.
\end{enumerate}
\end{obs}

\begin{figure}
  \scalebox{0.9}{\includegraphics{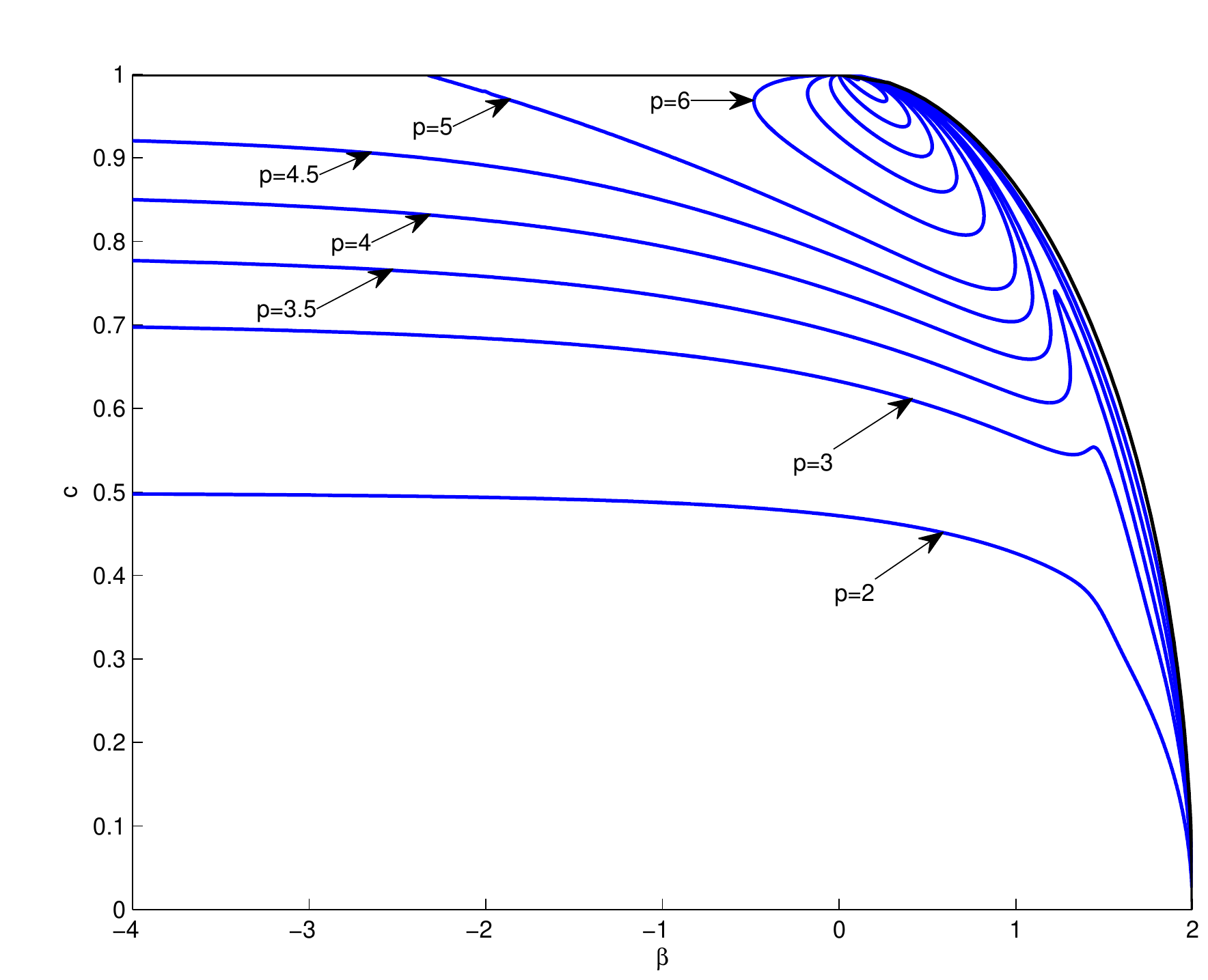}}
  \caption{Nodal sets of $d_{cc}$ for the odd nonlinearity $f(u)=|u|^{p-1}u$, for several values of $p$.}\label{F:dcc_nodal_odd}
\end{figure}

\begin{figure}
  \scalebox{0.9}{\includegraphics{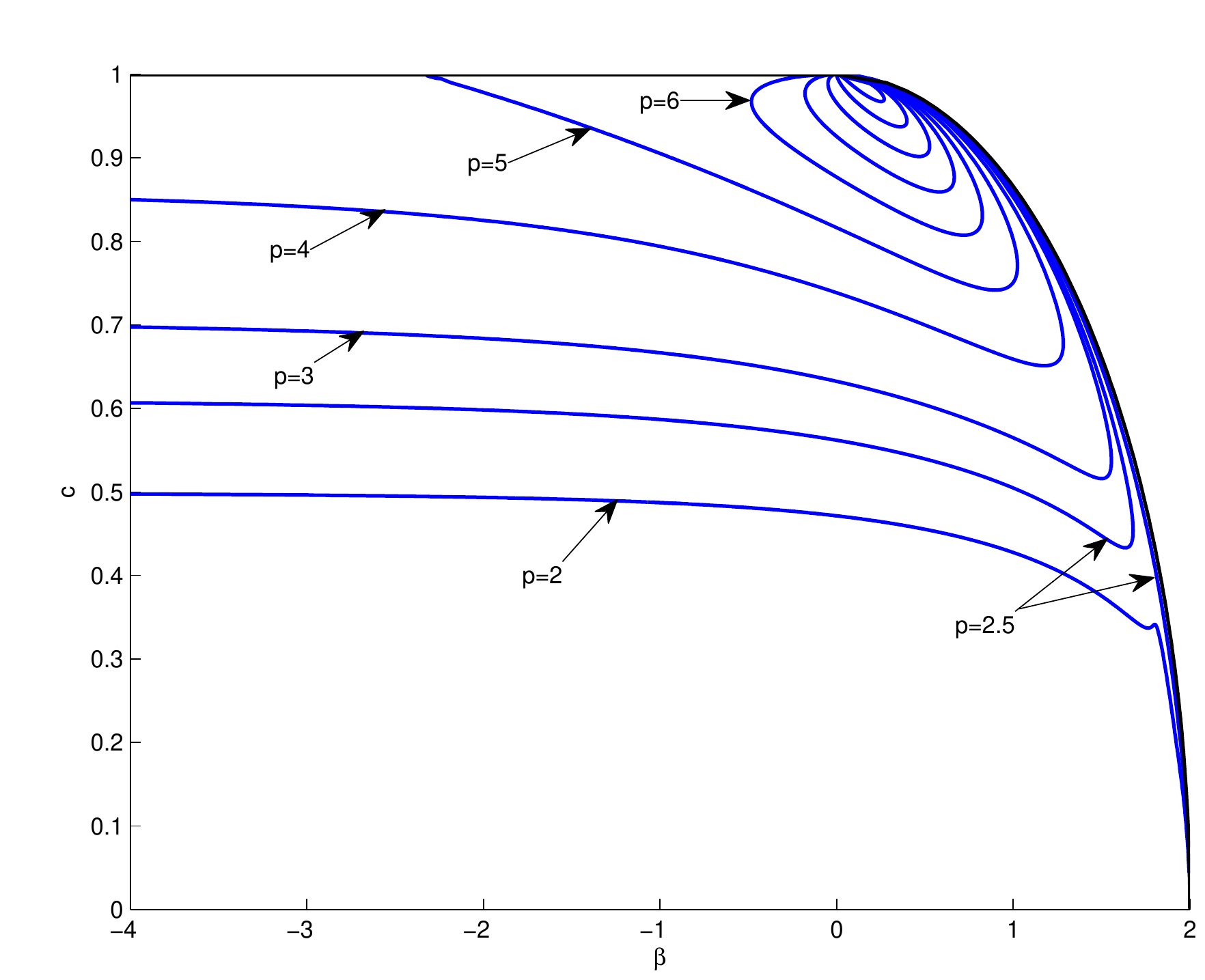}}
  \caption{Nodal sets of $d_{cc}$ for the even nonlinearity $f(u)=|u|^{p}$, for several values of $p$.}\label{F:dcc_nodal_even}
\end{figure}

\section*{}
\pdfbookmark[0]{References}{titr-1}

\end{document}